\documentclass[12pt,reqno]{amsart}

\usepackage[T1]{fontenc}
\usepackage[utf8]{inputenc}
\usepackage[expansion=false]{microtype}

\usepackage{amsmath,amssymb,amsthm}
\usepackage{mathtools}
\usepackage{bm}

\usepackage[a4paper,margin=1in]{geometry}

\usepackage[dvipsnames]{xcolor}
\usepackage[
 colorlinks = true,
 linkcolor = Mahogany,
 citecolor = ForestGreen,
 urlcolor = NavyBlue,
 pagebackref = true,
 bookmarks = true,
 pdfencoding = auto,
 psdextra
]{hyperref}
\usepackage{cleveref}

\theoremstyle{plain}
\newtheorem{theorem}{Theorem}[section]
\newtheorem{proposition}[theorem]{Proposition}
\newtheorem{lemma}[theorem]{Lemma}
\newtheorem{corollary}[theorem]{Corollary}

\newtheorem{theorema}{Theorem}

\theoremstyle{definition}

\theoremstyle{remark}
\newtheorem{remark}{Remark}[section]

\numberwithin{equation}{section}
\allowdisplaybreaks
\DeclareMathOperator{\curl}{curl}

\begin{document}

\title[A Fourier approach to higher and fractional order HUPs]
{A Fourier approach to the sharp stability of   Heisenberg
Uncertainty Principles of higher and fractional orders: a new perspective}

\author{Anh Xuan Do}
\address{Anh Xuan Do: Department of Mathematics, University of Connecticut, Storrs, CT
06269, USA}
\email{anh.do@uconn.edu}

\author{Nguyen Lam}
\address{Nguyen Lam: School of Science and the Environment, Grenfell Campus, Memorial
University of Newfoundland, Corner Brook, NL A2H5G4, Canada}
\email{nlam@mun.ca}

\author{Guozhen Lu}
\address{Guozhen Lu: Department of Mathematics, University of Connecticut, Storrs, CT
06269, USA}
\email{guozhen.lu@uconn.edu}

\author{Van Hoang Nguyen}
\address{Van Hoang Nguyen: Department of Mathematics, FPT University, Ha Noi, Vietnam}
\email{vanhoang0610@yahoo.com,~hoangnv47@fe.edu.vn}

\date{\today}


\begin{abstract}
We show that, on the Fourier side, the sharp second order Heisenberg Uncertainty
Principle (HUP) and its sharp stability for functions are nothing but the first order
$L^{2}$-Caffarelli--Kohn--Nirenberg inequality and its sharp stability, applied to their Fourier transforms. This gives a proof in a few lines of the sharp stability estimate that we
established recently by a much longer argument. The same observation then produces two
one-parameter families of sharp fractional Heisenberg Uncertainty Principles, in which
the gradient is replaced by a fractional power of the Laplacian: for each of them we
compute the sharp constants, characterize all the optimizers, and  establish the
sharp stability estimates.  The stability constant of the first family is equal to
$1$ for every nonnegative order and every dimension. Negative orders, where the
fractional Laplacian becomes a Riesz potential, are treated as well. Finally, we show
that the deficit carries much more information than the distance to the optimizers
alone: it controls an explicit chain of successive remainder terms whose constants are
the successive spectral gaps of an explicit operator. For the classical Heisenberg
Uncertainty Principle this gives three remainder terms with optimal explicit constants, and for the
second order HUP a chain of four remainder terms measured against explicit confluent
hypergeometric profiles and coupled through one single set of parameters. As applications, we also establish the chain of stability of the HUP for  curl-free vector fields.
\end{abstract}

\maketitle

%-----------------------------------------------------------------------------------

\section{Introduction and statement of the main results}

\subsection{The \texorpdfstring{$L^{2}$}{L2}-Caffarelli--Kohn--Nirenberg inequalities}

The starting point of this paper is the following $L^{2}$-Caffarelli--Kohn--Nirenberg
(CKN) inequality:
\begin{equation}\label{CKN}
\left( \int_{\mathbb{R}^{N}} \frac{|\nabla u|^{2}}{|x|^{2b}}\,dx \right)^{\frac{1}{2}}
\left( \int_{\mathbb{R}^{N}} \frac{|u|^{2}}{|x|^{2a}}\,dx \right)^{\frac{1}{2}}
\geq C(N,a,b)
\int_{\mathbb{R}^{N}} \frac{|u|^{2}}{|x|^{a+b+1}}\,dx,
\quad
u \in C_{0}^{\infty}(\mathbb{R}^{N} \setminus \{0\}).
\end{equation}
We note that the power $a+b+1$ in the right-hand integral is essential: it is precisely
the exponent dictated by the scaling invariance of \eqref{CKN}, and any other power
would break this invariance and prevent the inequality from holding with a uniform
constant. This family \eqref{CKN} encompasses several celebrated inequalities, such as
the Heisenberg Uncertainty Principle (corresponding to $a=-1$ and $b=0$), the hydrogen
uncertainty principle ($a=b=0$), and the classical Hardy inequality ($a=1$ and $b=0$),
each of which plays a fundamental role in analysis and in quantum mechanics. Moreover,
\eqref{CKN} is a special case of the more general Caffarelli--Kohn--Nirenberg
inequality, introduced by Caffarelli, Kohn, and Nirenberg in \cite{CKN}.

The optimal constant $C(N,a,b)>0$ in \eqref{CKN} was first investigated by Costa
\cite{Cos08} in a restricted parameter regime via the expanding-the-square method. This
analysis was subsequently extended to the full admissible range of parameters by
Catrina and Costa \cite{CC09}, using spherical harmonic decomposition and a Kelvin-type
transform. More recently, alternative direct proofs of the sharp constant $C(N,a,b)$,
valid over the entire parameter range and together with a complete characterization of
all optimizers, have been established in \cite{CFL21}. More precisely, the main
findings in \cite{CC09, CFL21, Cos08} can be summarized as follows.

\begin{theorema}\label{TA}
For all admissible parameters $(a,b)$, one has
\[
C(N,a,b)
= \max \left\{
\frac{|N-(a+b+1)|}{2}, \,
\frac{|N-(3b - a + 3)|}{2}
\right\}.
\]
Moreover, the sharp constant and the corresponding extremal functions are characterized
as follows.
\begin{enumerate}
\item If $(a,b) \in \mathcal{A}$, then
$C(N,a,b) = \frac{|N-(a+b+1)|}{2}$, and equality in \eqref{CKN} is attained precisely
by the functions
\[
u(x) = D \exp\!\left( \frac{t |x|^{\,b+1-a}}{b+1-a} \right),
\]
where $D \neq 0$, $t < 0$ in $\mathcal{A}_{1}$, and $t > 0$ in $\mathcal{A}_{2}$.

\item If $(a,b) \in \mathcal{B}$, then
$C(N,a,b) = \frac{|N-(3b - a + 3)|}{2}$, and equality in \eqref{CKN} is attained
precisely by the functions
\[
u(x) = D\, |x|^{\,2(b+1)-N}
\exp\!\left( \frac{t |x|^{\,b+1-a}}{b+1-a} \right),
\]
where $D \neq 0$, $t > 0$ in $\mathcal{B}_{1}$, and $t < 0$ in $\mathcal{B}_{2}$.

\item The only parameters for which the best constant is not attained are those on the
line $a=b+1$, where $C(N,b+1,b) = \frac{|N-2(b+1)|}{2}$.
\end{enumerate}
Here the regions are defined by
\[
\left\{
\begin{array}{ll}
\mathcal{A}_{1}:= \{ (a,b) \mid b+1-a > 0,\; b \le (N-2)/2 \}, \\[5pt]
\mathcal{A}_{2}:= \{ (a,b) \mid b+1-a < 0,\; b \ge (N-2)/2 \}, \\[5pt]
\mathcal{A}:= \mathcal{A}_{1} \cup \mathcal{A}_{2}, \\[5pt]
\mathcal{B}_{1}:= \{ (a,b) \mid b+1-a < 0,\; b \le (N-2)/2 \}, \\[5pt]
\mathcal{B}_{2}:= \{ (a,b) \mid b+1-a > 0,\; b \ge (N-2)/2 \}, \\[5pt]
\mathcal{B}:= \mathcal{B}_{1} \cup \mathcal{B}_{2}.
\end{array}
\right.
\]
\end{theorema}

We note that the two families of inequalities that we study in this paper correspond to
the parameters $(a,b)=(-s,0)$ with $s>-1$ and $(a,b)=(-s,-1)$ with $s>0$. In these
ranges, both of them belong to $\mathcal{A}_{1}$. For the negative fractional orders, they enter
the region $\mathcal{B}_{1}$, and, as we will see in Theorem \ref{Tfracneg} below, the
constant switches to the second branch of the maximum, while the optimizers change
shape. Both cases of Theorem \ref{TA} will therefore be used.

\subsection{Sharp stability of the \texorpdfstring{$L^{2}$}{L2}-CKN inequalities}

In \cite{BL85}, Brezis and Lieb raised a fundamental question concerning the sharp
Sobolev inequality in $W^{1,2}(\mathbb{R}^{N})$: can the \emph{deficit}, namely the
difference between the energy and the optimal constant times the critical
$L^{2^{*}}$-norm, be bounded below by a positive quantity comparable to the square of
the distance from $u$ to the manifold of optimizers? This question gave rise to a broad
and influential research program on the stability of functional and geometric
inequalities. Understanding such stability is important, since it not only refines
classical sharp inequalities by revealing how close a function is to achieving
equality, but also provides deep insights into the underlying geometry, rigidity
phenomena, and concentration behavior of extremal sequences. For interested readers, we
refer to \cite{BDNN20, BGKM25, BGKM25b, Carlen, CF13, CFMP09,  CLT23, CLT24, CLT242, CLT243,
CLTW, DEFFL, FN19, FZ22, LLR25, LLR26, VHN16, VHN19}; this list is far from
exhaustive.

Brezis and Lieb's question was solved affirmatively by Bianchi and Egnell \cite{BE91},
who established the stability estimate
\begin{equation}\label{StabilityS}
\int_{\mathbb{R}^{N}} |\nabla u|^{2}\,dx
- S_{N}\!\left(
\int_{\mathbb{R}^{N}} |u|^{\frac{2N}{N-2}}\,dx
\right)^{\!\frac{N-2}{N}}
\ge c_{\mathrm{BE}}\!
\inf_{U \in E_{\mathrm{Sob}}} \!
\int_{\mathbb{R}^{N}} |\nabla(u - U)|^{2}\,dx,
\end{equation}
where $S_{N}$ is the optimal Sobolev constant, $E_{\mathrm{Sob}}$ denotes the set of extremal
functions, and $c_{\mathrm{BE}}>0$ is a universal stability constant. This result shows that the
Sobolev deficit controls the distance in the gradient norm to the manifold of
optimizers, and that it is optimal in the underlying metrics and exponents.

However, the proof in \cite{BE91} is by contradiction, and therefore it does not
provide any information about the value of $c_{\mathrm{BE}}$. In fact, despite extensive
research, the value and the attainability of the stability constants have long remained
elusive. Let us describe here the situation for $c_{\mathrm{BE}}$, which is by now well
understood. On the one hand, a local analysis around the manifold of optimizers shows
that the best constant in the local version of \eqref{StabilityS} is the spectral gap
constant $\frac{4}{N+4}$; see \cite{CFW13}. On the other hand, K\"onig \cite{Kon23}
proved that $c_{\mathrm{BE}}$ is strictly smaller than $\frac{4}{N+4}$, and that it is attained.
In other words, the optimal constant in \eqref{StabilityS} is never given by the local
analysis. Quantitative lower bounds for $c_{\mathrm{BE}}$ were obtained only recently: Dolbeault,
Esteban, Figalli, Frank and Loss \cite{DEFFL} initiated a systematic analysis of the
Bianchi--Egnell constant, deriving explicit lower bounds for $c_{\mathrm{BE}}$ that are sharp in
the asymptotic regime $N \to \infty$, and establishing at the same time the global
stability of the Gaussian logarithmic Sobolev inequality via a gradient flow approach.
Subsequently, Chen, Lu and Tang \cite{CLT24, CLT242, CLT243} obtained explicit lower
bounds for the stability constants of the Hardy--Littlewood--Sobolev and of the higher
order and fractional Sobolev inequalities, with the optimal asymptotic behavior and
with dimension-dependent and order-dependent constants, which also lead to the global stability of Beckner's
logarithmic Sobolev inequality on the sphere. In a related development, Chen, Lu, Tang and Wang
\cite{CLTW} established the asymptotically sharp stability of the Sobolev inequality on
the Heisenberg group. Since the classical rearrangement and flow arguments are not
available in that setting, they introduced a new approach based on the CR Yamabe flow,
which allows one to pass from the local stability to the global one.

Let us now come back to the $L^{2}$-CKN inequality \eqref{CKN}. Its stability was first
addressed in
\cite{MV21} in the particular case $a=-1$ and $b=0$, corresponding to the Heisenberg
Uncertainty Principle (HUP). In their study, McCurdy and Venkatraman employed the
concentration--compactness framework, and therefore no information on the constants was
provided. A constructive approach was later developed by Fathi \cite{Fat21}, who
obtained explicit but non-optimal constants. Eventually, Cazacu, Flynn, Lam and Lu \cite{CFLL24}
established a sharp and fully quantitative stability estimate for the HUP: for all
$u \in C_{0}^{\infty}(\mathbb{R}^{N})$,
\begin{equation}\label{HUPstab}
\left(\int_{\mathbb{R}^{N}} |\nabla u|^{2}\,dx\right)^{\frac{1}{2}}
\left(\int_{\mathbb{R}^{N}} |x|^{2} |u|^{2}\,dx\right)^{\frac{1}{2}}
- \frac{N}{2} \int_{\mathbb{R}^{N}} |u|^{2}\,dx
\ge \inf_{v \in E_{\mathrm{HUP}}} \int_{\mathbb{R}^{N}} |u-v|^{2}\,dx,
\end{equation}
where
\[
E_{\mathrm{HUP}}:= \{\alpha e^{-\beta |x|^{2}}: \alpha \in \mathbb{R},\, \beta > 0\}
\]
denotes the family of Gaussian optimizers of the HUP. Moreover, the constant $1$ is
sharp and can be attained. See also Lam, Lu and Russanov \cite{LLR25, LLR26} for related improvements, and
Lam, Lodha, Lu and Sengupta \cite{LLLS25, LLLS26} for the Heisenberg Uncertainty Principle and for the whole
$L^{2}$-CKN family on half-spaces and orthants, together with their stability.

The approach in \cite{CFLL24} is based on two ideas: the explicit computation of the
remainder term in the HUP, and the use of the sharp Gaussian Poincar\'e inequality.
Because the latter is available only for log-concave measures, the stability estimates
derived in \cite{CFLL24} for the full family \eqref{CKN} were restricted to a rather
small set of parameters. Very recently, in \cite{DLLN26}, we removed all these
restrictions and computed the sharp stability constants for the whole family
\eqref{CKN}. More precisely, let us define the $L^{2}$-CKN deficit by
\begin{equation}\label{L2CKNdeficit}
\delta_{2}(v):= \left( \int_{\mathbb{R}^{N}} \frac{|\nabla v|^{2}}{|x|^{2b}}\,dx
\right)^{\frac{1}{2}}
\left( \int_{\mathbb{R}^{N}} \frac{|v|^{2}}{|x|^{2a}}\,dx \right)^{\frac{1}{2}}
- C(N,a,b) \int_{\mathbb{R}^{N}} \frac{|v|^{2}}{|x|^{a+b+1}}\,dx \geq 0.
\end{equation}
Then the following holds.

\begin{theorema}\label{TB}
Let $N \geq 1$ and $a,b \in \mathbb{R}$ with $b \leq \frac{N-2}{2}$, and define
\[
C_{\mathrm{PI}}(a,b,N) = \min\left\{ 2(1+b-a),\, \sqrt{(N-2-2b)^{2} + 4(N-1)} - (N-2-2b) \right\}.
\]
\begin{enumerate}
\item If $1+b-a>0$, then
\[
\delta_{2}(v) \geq \frac{C_{\mathrm{PI}}(a,b,N)}{2}
\inf_{c \in \mathbb{R},\, \lambda>0}
\int_{\mathbb{R}^{N}}
\frac{\left|v-c\,e^{-\frac{\lambda|x|^{b+1-a}}{b+1-a}}\right|^{2}}{|x|^{1+b+a}}\,dx.
\]
\item If $1+b-a<0$, then
\[
\delta_{2}(v) \geq \frac{C_{\mathrm{PI}}(-a+2b+2,b,N)}{2}
\inf_{c \in \mathbb{R},\, \lambda>0}
\int_{\mathbb{R}^{N}}
\frac{\left|v-c\,|x|^{2b+2-N}e^{\frac{\lambda|x|^{b+1-a}}{b+1-a}}\right|^{2}}
{|x|^{1+b+a}}\,dx.
\]
\end{enumerate}
In both cases the constant is sharp and can be attained. Similar statements, also with
explicit sharp constants, hold in the two remaining regimes $b \geq \frac{N-2}{2}$.
\end{theorema}

In \eqref{L2CKNdeficit} we do not display the dependence of $\delta_{2}$ on the
parameters $(a,b)$, since these will always be clear from the context. Let us also note
that Theorem \ref{TB} is stated above for real-valued $v$, but that the same statement
holds for complex-valued $v$, with the infimum taken over $c \in \mathbb{C}$. Indeed,
as we will see in Lemma \ref{LCKNidentity} and Proposition \ref{Pfourierstab}, the
proof relies only on a weighted
Poincar\'e inequality and on an exact remainder identity, both of which are valid for
complex-valued functions and complex parameters. It is worth noting that Theorem
\ref{TB} is a consequence of a family of sharp weighted $L^{2}$-Poincar\'e inequalities
with Gaussian type measures, which were also established in \cite{DLLN26}; see also Do, Flynn, Lam and Lu
\cite{DFLL23} for the corresponding $L^{p}$ theory. These weighted Poincar\'e inequalities play here the role that the classical
Gaussian Poincar\'e inequality plays in \cite{CFLL24}. We will use these tools
systematically in this paper.

\subsection{The second order Heisenberg Uncertainty Principle and our new approach.}\label{Spov}

Motivated by a question of Maz'ya \cite{Maz18} concerning the best constant in the HUP
for divergence-free vector fields, the authors in \cite{CFL22, CFL23} proved the
following sharp HUP for curl-free vector fields: for
$\mathbf{U} \in \left(C_{0}^{\infty}(\mathbb{R}^{N})\right)^{N}$ with
$\curl \mathbf{U}=0$, there holds
\begin{equation}\label{HUPcurl}
\left(\int_{\mathbb{R}^{N}} |\nabla \mathbf{U}|^{2}\,dx\right)^{\frac{1}{2}}
\left(\int_{\mathbb{R}^{N}} |x|^{2} |\mathbf{U}|^{2}\,dx\right)^{\frac{1}{2}}
\geq \frac{N+2}{2} \int_{\mathbb{R}^{N}} |\mathbf{U}|^{2}\,dx.
\end{equation}
Indeed, in this case we can write $\mathbf{U}=\nabla u$ for some scalar potential $u$,
and therefore \eqref{HUPcurl} is equivalent to the following second order Heisenberg
Uncertainty Principle:
\begin{equation}\label{2HUP}
\left(\int_{\mathbb{R}^{N}} |\Delta u|^{2}\,dx\right)^{\frac{1}{2}}
\left(\int_{\mathbb{R}^{N}} |x|^{2} |\nabla u|^{2}\,dx\right)^{\frac{1}{2}}
\geq \frac{N+2}{2} \int_{\mathbb{R}^{N}} |\nabla u|^{2}\,dx,
\end{equation}
where the constant $\frac{N+2}{2}$ improves the best constant $\frac{N}{2}$ that
corresponds to the scalar case. Moreover, the equality in \eqref{2HUP} is attained
exactly by the Gaussian profiles
\[
E_{\mathrm{SHUP}}:= \{\alpha e^{-\beta |x|^{2}}: \alpha \in \mathbb{R},\, \beta>0\}.
\]
Maz'ya's question itself was answered in dimension two in \cite{CFL22} and in all dimensions $N\ge3$ by Hamamoto \cite{Ham23}. The sharp stability of the resulting inequality for divergence-free fields is established by Do, Duong, Lam, Lu and Nguyen in \cite{DDLLNsol}.

Throughout this paper, $X$ denotes the completion of $C_{0}^{\infty}(\mathbb{R}^{N})$
under the seminorm
\[
\left(\int_{\mathbb{R}^{N}}|\Delta u|^{2}\,dx
+\int_{\mathbb{R}^{N}}|x|^{2}|\nabla u|^{2}\,dx\right)^{1/2},
\]
and we denote by
\[
\delta_{\mathrm{S}}(u):= \left(\int_{\mathbb{R}^{N}} |\Delta u|^{2}\,dx\right)^{\frac{1}{2}}
\left(\int_{\mathbb{R}^{N}} |x|^{2} |\nabla u|^{2}\,dx\right)^{\frac{1}{2}}
- \frac{N+2}{2} \int_{\mathbb{R}^{N}} |\nabla u|^{2}\,dx \geq 0
\]
the deficit of the second order HUP \eqref{2HUP}.

It is natural to ask about the stability of \eqref{2HUP}. This problem was first
studied by Duong and Nguyen in \cite{DN25}, by a spectral analysis of an Ornstein--Uhlenbeck type operator
associated with a genuinely second order remainder term, and then by Do, Lam and Lu in \cite{DoLL26},
where we established the sharp stability by a spherical harmonics decomposition combined
with a dimension-shifting identity for radial functions. In \cite{DoLL26}, the sharp
stability constant is the sharp constant of the first spherical mode, and the exact
identity used there to bound it from below by $\sqrt{N^2+4N-4}-N$ becomes an equality
for a Kummer function, so that this bound is its value. See also Duong and Nguyen \cite{DN23, DN26}, and
\cite{HY25}, where this value is also obtained by a linearization method.

\begin{theorema}\label{TC}
For all real-valued $u \in X$, one has
\begin{equation}\label{DoLLstab}
\delta_{\mathrm{S}}(u) \geq \frac{\sqrt{N^{2}+4N-4}-N}{2}
\inf_{u^{*} \in E_{\mathrm{SHUP}}}\int_{\mathbb{R}^{N}} |\nabla(u-u^{*})|^{2}\,dx ,
\end{equation}
and the constant $\frac{\sqrt{N^{2}+4N-4}-N}{2}$ is sharp and can be attained.
\end{theorema}

The proofs of Do, Lam and Lu \cite{DoLL26} and of Duong and Nguyen \cite{DN25} are, in different ways, rather involved,
because in both of them one has to work directly with a genuinely second order
quadratic form with a Gaussian weight. The starting point of this paper is the following
observation, which to the best of our knowledge has not been exploited before:
\emph{on the Fourier side, the second order HUP is a first order CKN inequality}. With
the normalization
\[
\mathcal{F}(f)(\xi)=\widehat{f}(\xi)
=\int_{\mathbb{R}^{N}} f(x)e^{-2\pi i x\cdot\xi}\,dx ,
\]
we will show in Lemma \ref{Ldictionary} that
\begin{equation}\label{dict1}
\int_{\mathbb{R}^{N}}|\xi|^{2\kappa}|\widehat{u}(\xi)|^{2}\,d\xi
=(2\pi)^{-2\kappa}\int_{\mathbb{R}^{N}}
\left|(-\Delta)^{\kappa/2}u\right|^{2}dx
\qquad (\kappa \in \mathbb{R}),
\end{equation}
where $(-\Delta)^{\kappa/2}$ is defined on the Fourier side by
$\mathcal{F}\left((-\Delta)^{\kappa/2}u\right)(\xi)=(2\pi|\xi|)^{\kappa}\widehat{u}(\xi)$,
and, what is more important,
\begin{equation}\label{dict2}
\int_{\mathbb{R}^{N}}|\nabla\widehat{u}(\xi)|^{2}\,d\xi
=(2\pi)^{2}\int_{\mathbb{R}^{N}}|x|^{2}|u|^{2}\,dx,
\qquad
\int_{\mathbb{R}^{N}}|\xi|^{2}|\nabla\widehat{u}(\xi)|^{2}\,d\xi
=\int_{\mathbb{R}^{N}}|x|^{2}|\nabla u|^{2}\,dx.
\end{equation}
The second identity in \eqref{dict2} says that the weight $|x|^{2}$ acting on
$\nabla u$ becomes, on the Fourier side, again a first order quantity, and that no
derivative is lost in the process. Combined with \eqref{dict1} for $\kappa=1$ and
$\kappa=2$, it gives
\begin{equation}\label{key}
\delta_{\mathrm{S}}(u)=(2\pi)^{2}\,\delta_{2}(\widehat{u}),
\end{equation}
where $\delta_{2}$ is the deficit \eqref{L2CKNdeficit} of the first order CKN inequality
with the parameters $(a,b)=(-2,-1)$. Since $C(N,-2,-1)=\frac{N+2}{2}$ by Theorem
\ref{TA}, the sharp second order HUP \eqref{2HUP} is a direct consequence of the first
order inequality \eqref{CKN}; and since
\[
C_{\mathrm{PI}}(-2,-1,N)=\min\left\{4,\, \sqrt{N^{2}+4N-4}-N\right\}=\sqrt{N^{2}+4N-4}-N
\]
by Lemma \ref{Lordering}, Theorem \ref{TB} gives another proof of Theorem \ref{TC}, with exactly the same sharp
constant, in a few lines; see Section \ref{Sproofs}.

The identities \eqref{dict1} and \eqref{dict2} are, however, much more than a device for
reproving Theorem \ref{TC}. The first factor of \eqref{CKN}, written for $\widehat{u}$,
is $\int_{\mathbb{R}^{N}}|\nabla\widehat{u}|^{2}|\xi|^{-2b}d\xi$, and \eqref{dict2} says
that it is a classical quantity in the physical variable exactly on the two lines $b=0$
and $b=-1$; we will show in Remark \ref{Ronlytwolines} that these are the only two such
lines. On the other hand, the parameter $a$ is completely free, since \eqref{dict1}
holds for every real $\kappa$. Applying \eqref{CKN} and Theorem \ref{TB} to $\widehat{u}$
with $a=-s$ and $b \in \{0,-1\}$, we therefore obtain two one-parameter families of
sharp uncertainty principles in which the order $s$ of the derivative is
\emph{fractional}, together with their optimizers and their sharp stability. The second
order HUP is the point $s=2$ of the second family, and the classical one is the point
$s=1$ of the first. This is the content of the next subsection.

\subsection{Fractional Heisenberg Uncertainty Principles}

Throughout this paper, for $\sigma>0$ and
$\lambda>0$ we denote by
\begin{equation}\label{profile}
G_{\sigma,\lambda}:=\mathcal{F}^{-1}
\left(e^{-\frac{\lambda|\xi|^{\sigma}}{\sigma}}\right)
\end{equation}
the inverse Fourier transform of the stretched exponential. We note that
$G_{\sigma,\lambda}$ is real-valued and radial, and that $G_{2,\lambda}$ is a Gaussian.

The first of the two families corresponds to the line $b=0$. In it, only the order of
the derivative is fractional, while the second moment
$\int_{\mathbb{R}^{N}}|x|^{2}|u|^{2}\,dx$ is the classical one; for this reason we refer
to it as the fractional Heisenberg Uncertainty Principle.

\begin{theorem}\label{Tfrac1}
Let $N \geq 2$ and $s>-1$. Then for all $u \in C_{0}^{\infty}(\mathbb{R}^{N})$,
\begin{equation}\label{frac1}
\left(\int_{\mathbb{R}^{N}}\left|(-\Delta)^{s/2}u\right|^{2}dx\right)^{\frac{1}{2}}
\left(\int_{\mathbb{R}^{N}}|x|^{2}|u|^{2}\,dx\right)^{\frac{1}{2}}
\geq \frac{N+s-1}{2}
\int_{\mathbb{R}^{N}}\left|(-\Delta)^{\frac{s-1}{4}}u\right|^{2}dx,
\end{equation}
and the constant $\frac{N+s-1}{2}$ is sharp. Moreover, the equality in \eqref{frac1}
is attained exactly by the profiles $\alpha\,G_{1+s,\lambda}$ with
$\alpha \in \mathbb{C}\setminus\{0\}$ and $\lambda>0$, and one has the sharp stability estimate
\begin{equation}\label{frac1stab}
\begin{aligned}
&\left(\int_{\mathbb{R}^{N}}\left|(-\Delta)^{s/2}u\right|^{2}dx\right)^{\frac{1}{2}}
\left(\int_{\mathbb{R}^{N}}|x|^{2}|u|^{2}\,dx\right)^{\frac{1}{2}}
-\frac{N+s-1}{2}
\int_{\mathbb{R}^{N}}\left|(-\Delta)^{\frac{s-1}{4}}u\right|^{2}dx\\
&\qquad\geq
\min\left\{1+s,\,1\right\}
\inf_{\alpha \in \mathbb{C},\, \lambda>0}
\int_{\mathbb{R}^{N}}\left|(-\Delta)^{\frac{s-1}{4}}
\left(u-\alpha\,G_{1+s,\lambda}\right)\right|^{2}dx,
\end{aligned}
\end{equation}
where the constant $\min\left\{1+s,1\right\}$ is sharp and can be attained. In
particular, the sharp stability constant is equal to $1$ for every $s \geq 0$.
\end{theorem}

It is worth noting that the stability constant in \eqref{frac1stab} does not depend on
the dimension, and that it is equal to $1$ for every $s \geq 0$. In other words, on the
whole range $s \geq 0$, the fractional family \eqref{frac1} enjoys exactly the same
dimension-free sharp stability constant as the classical Heisenberg Uncertainty
Principle, while for $-1<s<0$ the constant degenerates linearly as $s$ approaches the
Hardy line $s=-1$. As we will see in Section \ref{Sproofs}, the reason for this
rigidity is the algebraic identity
\begin{equation}\label{magic}
(N-2)^{2}+4(N-1)=N^{2},
\end{equation}
which forces the second quantity in the definition of $C_{\mathrm{PI}}(-s,0,N)$ to be equal to
$2$, independently of $s$ and of $N$, so that $C_{\mathrm{PI}}(-s,0,N)=\min\{2(1+s),2\}$ for
every $s>-1$.
When $s=1$, the operator $(-\Delta)^{\frac{s-1}{4}}$ is the identity,
$\int_{\mathbb{R}^{N}}|(-\Delta)^{1/2}u|^{2}dx=\int_{\mathbb{R}^{N}}|\nabla u|^{2}dx$,
and $G_{2,\lambda}$ is a Gaussian, so that \eqref{frac1} is the classical Heisenberg
Uncertainty Principle and \eqref{frac1stab} is exactly the sharp stability estimate
\eqref{HUPstab} of \cite{CFLL24}.

The second family corresponds to the line $b=-1$. Here the second factor is the
weighted Dirichlet integral $\int_{\mathbb{R}^{N}}|x|^{2}|\nabla u|^{2}\,dx$ instead of
the second moment. This family interpolates between the first order and the second
order uncertainty principles.

\begin{theorem}\label{Tfrac2}
Let $N \geq 2$ and $s>0$. Then for all $u \in C_{0}^{\infty}(\mathbb{R}^{N})$,
\begin{equation}\label{frac2}
\left(\int_{\mathbb{R}^{N}}\left|(-\Delta)^{s/2}u\right|^{2}dx\right)^{\frac{1}{2}}
\left(\int_{\mathbb{R}^{N}}|x|^{2}|\nabla u|^{2}\,dx\right)^{\frac{1}{2}}
\geq \frac{N+s}{2}
\int_{\mathbb{R}^{N}}\left|(-\Delta)^{s/4}u\right|^{2}dx,
\end{equation}
and the constant $\frac{N+s}{2}$ is sharp. Moreover, the equality in \eqref{frac2} is
attained exactly by the profiles $\alpha\,G_{s,\lambda}$ with
$\alpha \in \mathbb{C}\setminus\{0\}$ and $\lambda>0$, and one has the sharp stability estimate
\begin{equation}\label{frac2stab}
\begin{aligned}
&\left(\int_{\mathbb{R}^{N}}\left|(-\Delta)^{s/2}u\right|^{2}dx\right)^{\frac{1}{2}}
\left(\int_{\mathbb{R}^{N}}|x|^{2}|\nabla u|^{2}\,dx\right)^{\frac{1}{2}}
-\frac{N+s}{2}
\int_{\mathbb{R}^{N}}\left|(-\Delta)^{s/4}u\right|^{2}dx\\
&\qquad\geq
\frac{\min\left\{2s,\, \sqrt{N^{2}+4N-4}-N\right\}}{2}
\inf_{\alpha \in \mathbb{C},\, \lambda>0}
\int_{\mathbb{R}^{N}}\left|(-\Delta)^{s/4}
\left(u-\alpha\,G_{s,\lambda}\right)\right|^{2}dx,
\end{aligned}
\end{equation}
and this constant is also sharp.
\end{theorem}

In contrast with Theorem \ref{Tfrac1}, the stability constant in \eqref{frac2stab}
depends on $s$ and on $N$: it is equal to $s$ when $2s \leq \sqrt{N^{2}+4N-4}-N$, and
it stabilizes at $\frac{\sqrt{N^{2}+4N-4}-N}{2}$ for the large values of $s$. When
$s=2$, \eqref{frac2} is \eqref{2HUP} and \eqref{frac2stab} is \eqref{DoLLstab}.

We close this subsection with the negative fractional orders. In Theorem \ref{Tfrac1},
the range $s>-1$ is exactly the range for which the parameters $(-s,0)$ stay in the
region $\mathcal{A}_{1}$. It is natural to ask what happens beyond this range, and, for
the family \eqref{frac2}, what happens for $s<0$. The answer is provided by the second
case of Theorem \ref{TA} and of Theorem \ref{TB}: the parameters enter the region
$\mathcal{B}_{1}$, the constant switches to the second branch of the maximum in Theorem
\ref{TA}, and the optimizers acquire the singular factor $|x|^{2(b+1)-N}$. In order to
state the result, for $\sigma>0$, $\lambda>0$ and $\kappa \in \mathbb{R}$ we introduce
the profiles
\begin{equation}\label{profileB}
G_{\sigma,\lambda}^{\kappa}
:=\mathcal{F}^{-1}\left(|\xi|^{\kappa}\,
e^{-\frac{\lambda|\xi|^{-\sigma}}{\sigma}}\right),
\end{equation}
which are again real-valued and radial. We note that the Fourier transform of
$G_{\sigma,\lambda}^{\kappa}$ vanishes at the origin faster than any power of $|\xi|$,
and behaves like $|\xi|^{\kappa}$ at infinity.

\begin{theorem}\label{Tfracneg}
\begin{enumerate}
\item Let $N \geq 2$ and $-\frac{N}{2}<s<0$. Then for all $u \in C_{0}^{\infty}(\mathbb{R}^{N})$,
\begin{equation}\label{frac2neg}
\left(\int_{\mathbb{R}^{N}}\left|(-\Delta)^{s/2}u\right|^{2}dx\right)^{\frac{1}{2}}
\left(\int_{\mathbb{R}^{N}}|x|^{2}|\nabla u|^{2}\,dx\right)^{\frac{1}{2}}
\geq \frac{N-s}{2}
\int_{\mathbb{R}^{N}}\left|(-\Delta)^{s/4}u\right|^{2}dx,
\end{equation}
and the constant $\frac{N-s}{2}$ is sharp; the equality is attained, in the completion
of $C_{0}^{\infty}(\mathbb{R}^{N})$ under the corresponding norms, exactly by the
profiles $\alpha\,G_{-s,\lambda}^{-N}$ with $\alpha \in \mathbb{C}\setminus\{0\}$ and
$\lambda>0$. Moreover,
\begin{equation}\label{frac2negstab}
\begin{aligned}
&\left(\int_{\mathbb{R}^{N}}\left|(-\Delta)^{s/2}u\right|^{2}dx\right)^{\frac{1}{2}}
\left(\int_{\mathbb{R}^{N}}|x|^{2}|\nabla u|^{2}\,dx\right)^{\frac{1}{2}}
-\frac{N-s}{2}
\int_{\mathbb{R}^{N}}\left|(-\Delta)^{s/4}u\right|^{2}dx\\
&\qquad\geq
\frac{\min\left\{-2s,\, \sqrt{N^{2}+4N-4}-N\right\}}{2}
\inf_{\alpha \in \mathbb{C},\, \lambda>0}
\int_{\mathbb{R}^{N}}\left|(-\Delta)^{s/4}
\left(u-\alpha\,G_{-s,\lambda}^{-N}\right)\right|^{2}dx,
\end{aligned}
\end{equation}
and this constant is also sharp.
\item Let $N \geq 3$ and $-\frac{N}{2}<s<-1$. Then for all
$u \in C_{0}^{\infty}(\mathbb{R}^{N})$,
\begin{equation}\label{frac1neg}
\left(\int_{\mathbb{R}^{N}}\left|(-\Delta)^{s/2}u\right|^{2}dx\right)^{\frac{1}{2}}
\left(\int_{\mathbb{R}^{N}}|x|^{2}|u|^{2}\,dx\right)^{\frac{1}{2}}
\geq \frac{N-3-s}{2}
\int_{\mathbb{R}^{N}}\left|(-\Delta)^{\frac{s-1}{4}}u\right|^{2}dx,
\end{equation}
and the constant $\frac{N-3-s}{2}$ is sharp; the equality is attained, in the same
sense as above, exactly by the profiles $\alpha\,G_{-(1+s),\lambda}^{2-N}$ with
$\alpha \in \mathbb{C}\setminus\{0\}$ and $\lambda>0$. Moreover,
\begin{equation}\label{frac1negstab}
\begin{aligned}
&\left(\int_{\mathbb{R}^{N}}\left|(-\Delta)^{s/2}u\right|^{2}dx\right)^{\frac{1}{2}}
\left(\int_{\mathbb{R}^{N}}|x|^{2}|u|^{2}\,dx\right)^{\frac{1}{2}}
-\frac{N-3-s}{2}
\int_{\mathbb{R}^{N}}\left|(-\Delta)^{\frac{s-1}{4}}u\right|^{2}dx\\
&\qquad\geq
\min\left\{|1+s|,\,1\right\}
\inf_{\alpha \in \mathbb{C},\, \lambda>0}
\int_{\mathbb{R}^{N}}\left|(-\Delta)^{\frac{s-1}{4}}
\left(u-\alpha\,G_{-(1+s),\lambda}^{2-N}\right)\right|^{2}dx,
\end{aligned}
\end{equation}
and this constant is also sharp.
\end{enumerate}
\end{theorem}

Combining \eqref{frac2} and \eqref{frac2neg}, the sharp constant of the second family
is $\frac{N+|s|}{2}$ for all admissible $s$, and, combining \eqref{frac1} and
\eqref{frac1neg}, the sharp constant of the first family is $\frac{N-2+|1+s|}{2}$; both expressions are continuous in $s$, in agreement with the
maximum in Theorem \ref{TA}. The borderline values $s=0$ and $s=-1$ correspond to the
line $a=b+1$ of Theorem \ref{TA}, that is, to weighted Hardy inequalities: the
inequalities remain valid there with the constants $\frac{N}{2}$ and $\frac{N-2}{2}$,
but these are not attained. The restriction $s>-\frac{N}{2}$ is necessary, since for a
generic $u \in C_{0}^{\infty}(\mathbb{R}^{N})$ one has $\widehat{u}(0)\neq0$, so that
$\int_{\mathbb{R}^{N}}|(-\Delta)^{s/2}u|^{2}dx<\infty$ precisely when $s>-\frac{N}{2}$.
Finally, for $N=2$ the range in the second part is empty, in agreement with the fact
that $\frac{N-2+|1+s|}{2}$ vanishes at $(N,s)=(2,-1)$.

\subsection{Chains of remainder terms}

Before stating the next results, let us fix some notation that will be used throughout
the paper. We set
\begin{equation}\label{Dk}
c_{k}:=k(N+k-2),\qquad
D_{k}:=\sqrt{N^{2}+4c_{k}}-N\quad (k \geq 0),
\end{equation}
so that
\[
D_{1}=\sqrt{N^{2}+4N-4}-N,\qquad
D_{2}=\sqrt{N^{2}+8N}-N,\qquad
D_{3}=\sqrt{N^{2}+12N+12}-N.
\]
We note that $D_{k}$ is exactly the quantity
$\sqrt{(N-2-2b)^{2}+4c_{k}}-(N-2-2b)$ from \cite{DLLN26} evaluated at $b=-1$. We will
check in Lemma \ref{Lordering} that for every $N \geq 2$,
\begin{equation}\label{ordering}
0<D_{1}<D_{2}<4<D_{3}
\qquad\text{and}\qquad
D_{3}-4<D_{1}.
\end{equation}
Moreover, let $\{\phi_{k,i}\}_{i=1}^{N_{k}}$ be an orthonormal basis of
$L^{2}(\mathbb{S}^{N-1})$ consisting of the spherical harmonics of degree $k$, so that
$-\Delta_{\mathbb{S}^{N-1}}\phi_{k,i}=c_{k}\phi_{k,i}$, and let us write
\[
\mathbf{\Phi}_{k}:=\left(\phi_{k,1},\ldots,\phi_{k,N_{k}}\right)
\]
for the vector of all the spherical harmonics of degree $k$. When no confusion can
arise, we will write $\mathbf{\Phi}_{k}(x)$ instead of
$\mathbf{\Phi}_{k}\left(x/|x|\right)$.

So far, each of the two fractional families has been measured against a single family
of optimizers. Our next result shows that, on the line $b=0$, the estimate
\eqref{frac1stab} is only the first term of a chain whose length grows with $s$, and
whose constants add up to $\lfloor s\rfloor+2$; when $s$ is an integer, all of them are
again equal to $1$. The mechanism is an exact algebraic coincidence: on this line one
has
\begin{equation}\label{Dk2k}
(N-2)^{2}+4c_{k}=(N-2+2k)^{2},
\qquad\text{that is,}\qquad
D_{k}^{(0)}:=\sqrt{(N-2)^{2}+4c_{k}}-(N-2)=2k
\end{equation}
for every $k \geq 0$ and every $N$, where $c_{k}=k(N+k-2)$ as in \eqref{Dk}. In
particular, the spherical harmonic profile $|\xi|^{D_{k}^{(0)}/2}\phi_{k,i}(\xi/|\xi|)$
is the solid harmonic of degree $k$, and the competitors on the Fourier side are simply
harmonic polynomials.

Accordingly, for $j \geq 0$ we denote by $\mathcal{Y}_{j}$ the space of the harmonic
polynomials on $\mathbb{R}^{N}$ of degree at most $j$, with complex coefficients, and we
set $\sigma:=1+s$ and
\begin{equation}\label{Gcal}
\begin{aligned}
\mathcal{G}_{j}^{s}&:=\left\{\mathcal{F}^{-1}
\left(P\,e^{-\frac{\lambda|\xi|^{\sigma}}{\sigma}}\right)
\:\ P \in \mathcal{Y}_{j},\ \lambda>0\right\},
\qquad j=0,1,\ldots,m,\\
\mathcal{G}_{m+1}^{s}&:=\left\{\mathcal{F}^{-1}
\left(\left(P+d|\xi|^{\sigma}\right)e^{-\frac{\lambda|\xi|^{\sigma}}{\sigma}}\right)
\:\ P \in \mathcal{Y}_{m},\ d \in \mathbb{C},\ \lambda>0\right\},
\end{aligned}
\end{equation}
where $m:=\lfloor s\rfloor+1$. By the Hecke--Bochner formula, each element of
$\mathcal{G}_{j}^{s}$ is a sum of terms of the form (solid harmonic of degree
$k$)$\times$(radial function); we note that
$\mathcal{G}_{0}^{s}=\{\alpha\,G_{1+s,\lambda}\}$ is exactly the family of the
optimizers of Theorem \ref{Tfrac1}, and that
$\mathcal{G}_{0}^{s}\subset\mathcal{G}_{1}^{s}\subset\cdots
\subset\mathcal{G}_{m+1}^{s}$.

\begin{theorem}\label{Tfraccascade}
Let $N \geq 2$ and $s>0$, and write $m=\lfloor s\rfloor+1$ and
$\{s\}=s-\lfloor s\rfloor$. Then, for all $u \in C_{0}^{\infty}(\mathbb{R}^{N})$,
\begin{equation}\label{fraccascade}
\begin{aligned}
&\left(\int_{\mathbb{R}^{N}}\left|(-\Delta)^{s/2}u\right|^{2}dx\right)^{\frac{1}{2}}
\left(\int_{\mathbb{R}^{N}}|x|^{2}|u|^{2}\,dx\right)^{\frac{1}{2}}
-\frac{N+s-1}{2}
\int_{\mathbb{R}^{N}}\left|(-\Delta)^{\frac{s-1}{4}}u\right|^{2}dx\\
&\qquad\geq\ \sum_{j=0}^{m-1}d_{s}^{2}\!\left(u,\mathcal{G}_{j}^{s}\right)
\ +\ \{s\}\,d_{s}^{2}\!\left(u,\mathcal{G}_{m}^{s}\right)
\ +\ \left(1-\{s\}\right)d_{s}^{2}\!\left(u,\mathcal{G}_{m+1}^{s}\right),
\end{aligned}
\end{equation}
where
\[
d_{s}^{2}(u,\mathcal{G}):=\inf_{w \in \mathcal{G}}
\int_{\mathbb{R}^{N}}\left|(-\Delta)^{\frac{s-1}{4}}(u-w)\right|^{2}dx.
\]
In particular, when $s$ is an integer, \eqref{fraccascade} consists of exactly
$s+2$ terms, and all of its constants are equal to $1$. In general, the constants in
\eqref{fraccascade} add up to $\lfloor s\rfloor+2$, and they cannot be improved, with the possible
exception of the constant $\{s\}$ of the $m$-th term, which vanishes when $s$ is an
integer; see Remark \ref{Rspectral}.
\end{theorem}

We emphasize that the first term of \eqref{fraccascade} is precisely
\eqref{frac1stab}, so that Theorem \ref{Tfraccascade} is a genuine refinement of
Theorem \ref{Tfrac1}, and that the number of terms is unbounded as $s \to \infty$. It
is also worth noting that the constants do not depend on the dimension.

We state separately the case $s=1$, since it concerns the classical Heisenberg
Uncertainty Principle. When $s=1$ we have $\sigma=2$, and the Hecke--Bochner
formula becomes the classical Hecke identity: if $Y_{k}$ is a solid harmonic of degree
$k$, then the inverse Fourier transform of $Y_{k}(\xi)e^{-\lambda|\xi|^{2}/2}$ is a
nonzero constant multiple of $Y_{k}(x)$ times a Gaussian, the constant depending on $k$
and on $\lambda$. Since the coefficients are free, the sets $\mathcal{G}_{j}^{1}$ are
therefore obtained by multiplying a Gaussian by an arbitrary harmonic polynomial of
degree at most $j$. Moreover $\mathcal{Y}_{2}$ together with $|\xi|^{2}$ spans all the
polynomials of degree at most $2$. Therefore, denoting by
\[
\mathcal{Q}_{j}:=\left\{Q(x)\,e^{-\beta|x|^{2}}\:\
Q \text{ a polynomial of degree at most } j,\ \beta>0\right\},
\]
so that $\mathcal{Q}_{0}=E_{\mathrm{HUP}}$ is the family of the Gaussian optimizers, we obtain
the following improvement of \eqref{HUPstab}.

\begin{corollary}\label{CHUPcascade}
Let $N \geq 2$. For all $u \in C_{0}^{\infty}(\mathbb{R}^{N})$,
\begin{equation}\label{HUPcascade}
\begin{aligned}
&\left(\int_{\mathbb{R}^{N}} |\nabla u|^{2}\,dx\right)^{\frac{1}{2}}
\left(\int_{\mathbb{R}^{N}} |x|^{2} |u|^{2}\,dx\right)^{\frac{1}{2}}
- \frac{N}{2} \int_{\mathbb{R}^{N}} |u|^{2}\,dx\\
&\qquad\geq
\inf_{w \in \mathcal{Q}_{0}}\int_{\mathbb{R}^{N}}|u-w|^{2}dx
+\inf_{w \in \mathcal{Q}_{1}}\int_{\mathbb{R}^{N}}|u-w|^{2}dx
+\inf_{w \in \mathcal{Q}_{2}}\int_{\mathbb{R}^{N}}|u-w|^{2}dx.
\end{aligned}
\end{equation}
\end{corollary}

The first term of \eqref{HUPcascade} is exactly the sharp stability estimate
\eqref{HUPstab} of \cite{CFLL24}, and the first two terms are the corresponding
statement of \cite{DLLN26}. The third term is new. In words: the Heisenberg deficit
controls, with constant $1$ in each case, the distance to the Gaussians, the distance
to the Gaussians multiplied by a polynomial of degree at most $1$, and the distance to
the Gaussians multiplied by a polynomial of degree at most $2$.

The same mechanism applies to the second family. Here the relevant constants are the
$D_{k}$ of \eqref{Dk}, which are no longer integers, and consequently the statement is
less rigid: the length of the chain depends on the dimension as well as on $s$, and
the constants are no longer equal to $1$. For $\lambda>0$, $j \geq 0$ and
$P \in \mathcal{Y}_{j}^{\ast}$, where
\[
\mathcal{Y}_{j}^{\ast}:=\left\{\sum_{k=0}^{j}|\xi|^{\frac{D_{k}}{2}}
\mathbf{d}_{k}\cdot\mathbf{\Phi}_{k}(\xi/|\xi|)\:\
\mathbf{d}_{k}\in\mathbb{C}^{N_{k}}\right\}
\]
denotes the space generated by the profiles $|\xi|^{D_{k}/2}\phi_{k,i}$ of degree at
most $j$, we set
\begin{equation}\label{Hcal}
\begin{aligned}
\mathcal{H}_{j}^{s}&:=\left\{\mathcal{F}^{-1}
\left(P\,e^{-\frac{\lambda|\xi|^{s}}{s}}\right)
\:\ P \in \mathcal{Y}_{j}^{\ast},\ \lambda>0\right\},\\
\mathcal{H}_{m+1}^{s}&:=\left\{\mathcal{F}^{-1}
\left(\left(P+d|\xi|^{s}\right)e^{-\frac{\lambda|\xi|^{s}}{s}}\right)
\:\ P \in \mathcal{Y}_{m}^{\ast},\ d \in \mathbb{C},\ \lambda>0\right\}.
\end{aligned}
\end{equation}
Again $\mathcal{H}_{0}^{s}=\{\alpha\,G_{s,\lambda}\}$ is the family of the optimizers of
Theorem \ref{Tfrac2}, and the families are nested.

\begin{theorem}\label{Tfrac2cascade}
Let $N \geq 2$ and let $s \geq \frac{D_{2}-D_{1}}{2}$. Let $m \geq 1$ be the largest integer
such that
\begin{equation}\label{mcond}
D_{m}\leq 2s
\qquad\text{and}\qquad
D_{m+1}-D_{1}\leq 2s,
\end{equation}
and set
\[
\begin{aligned}
&\beta_{j}:=D_{j+1}-D_{j}\quad (0\leq j\leq m-1),\\
&\beta_{m}:=\min\left\{2s,\,D_{m+1}\right\}-D_{m},
\qquad
\beta_{m+1}:=\left(D_{m+1}-2s\right)^{+},
\end{aligned}
\]

Then, for all $u \in C_{0}^{\infty}(\mathbb{R}^{N})$,
\begin{equation}\label{frac2cascade}
\begin{aligned}
&\left(\int_{\mathbb{R}^{N}}\left|(-\Delta)^{s/2}u\right|^{2}dx\right)^{\frac{1}{2}}
\left(\int_{\mathbb{R}^{N}}|x|^{2}|\nabla u|^{2}\,dx\right)^{\frac{1}{2}}
-\frac{N+s}{2}
\int_{\mathbb{R}^{N}}\left|(-\Delta)^{s/4}u\right|^{2}dx\\
&\qquad\geq\ \sum_{j=0}^{m+1}\frac{\beta_{j}}{2}
\inf_{w \in \mathcal{H}_{j}^{s}}\int_{\mathbb{R}^{N}}
\left|(-\Delta)^{s/4}(u-w)\right|^{2}dx.
\end{aligned}
\end{equation}
The constants add up to $\frac{D_{m+1}}{2}$, and the first of them,
$\frac{\beta_{0}}{2}=\frac{D_{1}}{2}$, is the sharp constant of \eqref{frac2stab}. In
fact, the constants in \eqref{frac2cascade} cannot be improved, with the possible
exception of $\beta_{m}$; see Remark \ref{Rspectral}.
\end{theorem}

We note that the assumption $s \geq \frac{D_{2}-D_{1}}{2}$ guarantees that $m=1$ is
admissible in \eqref{mcond}, since $D_{1}<D_{2}-D_{1}$ for every $N \geq 2$ by
\eqref{ordering2} of Lemma \ref{Lordering}. It is
therefore stronger than the condition $C_{\mathrm{PI}}(-s,-1,N)=D_{1}$, which only says that the
spherical harmonic mode of degree $1$, and not the second radial mode, is the first one
to be removed. In the complementary range $0<s<\frac{D_{2}-D_{1}}{2}$, a chain still
holds, with a different structure; we do not write it down here. We also point out that, in contrast with Theorem \ref{Tfraccascade}, the
number $m$ in \eqref{mcond} depends on the dimension: for instance, when
$s=\frac{7}{4}$ one has $m=2$ for $2 \leq N \leq 7$, while $m=1$ for every $N \geq 8$.

We now return to the second order HUP, that is, to the value $s=2$ of the second
family. Here we are in a much more favorable situation than the one of the
Bianchi--Egnell inequality \eqref{StabilityS}, where the sharp constant
$c_{\mathrm{BE}}$ is still known only through upper and lower bounds: the constant in
\eqref{DoLLstab} is explicit, and we know exactly the functions for which the equality
holds. Indeed, \eqref{DoLLstab} is only the first term of the chain given by Theorem
\ref{Tfrac2cascade} with $s=2$, where $m=2$ and the four constants are
$\frac{D_{1}}{2}$, $\frac{D_{2}-D_{1}}{2}$, $\frac{4-D_{2}}{2}$ and $\frac{D_{3}-4}{2}$
by Lemma \ref{Lordering}. Since in this case $\sigma=2$, the Hecke--Bochner formula can
be inverted explicitly, and the competitor families become explicit families of
confluent hypergeometric profiles in the physical variable. Moreover, the four remainder
terms can be coupled, that is, measured with one single set of parameters, which is a
genuine strengthening that is not contained in Theorem \ref{Tfrac2cascade}. Its
counterpart on the Fourier side is a chain of weighted Poincar\'e inequalities for the
measure $e^{-|\xi|^{2}}|\xi|^{2}d\xi$, stated as Theorem \ref{TimprovedPoincare} in
Subsection \ref{Schains}; its first two layers recover the corresponding statement of
\cite[Section 2]{DLLN26}, while the third and the fourth ones are new. In order to state the result in the physical variable, we introduce
the two radial profiles
\begin{equation}\label{Theta}
\begin{aligned}
\Theta_{1}(x;\beta)
&:={}_{1}F_{1}\!\left(\frac{D_{1}}{4}+\frac{N+1}{2},\ \frac{N+2}{2};\
-\beta|x|^{2}\right),\\[3pt]
\Theta_{2}(x;\beta)
&:={}_{1}F_{1}\!\left(\frac{D_{2}}{4}+\frac{N}{2}+1,\ \frac{N}{2}+2;\
-\beta|x|^{2}\right),
\end{aligned}
\end{equation}
where ${}_{1}F_{1}$ denotes the confluent hypergeometric function of Kummer, together
with the second radial Gaussian profile
\begin{equation}\label{Lprofile}
\mathcal{L}(x;\beta):=\left(2\beta|x|^{2}-\frac{N-2}{2}\right)e^{-\beta|x|^{2}}.
\end{equation}

\begin{theorem}\label{Tcascade}
Let $N \geq 2$ and let $u \in X$, real- or complex-valued. For
$P=(\alpha,\tau,\mathbf{a},\mathbf{b},\beta) \in \mathbb{C}\times\mathbb{C}
\times\mathbb{C}^{N}\times\mathbb{C}^{N_{2}}\times(0,\infty)$, set
\begin{equation}\label{wj}
\begin{aligned}
w_{0}(\cdot;P)&:=\alpha e^{-\beta|x|^{2}},
&w_{1}(\cdot;P)&:=w_{0}+\Theta_{1}(x;\beta)\,\mathbf{a}\cdot x,\\
w_{2}(\cdot;P)&:=w_{1}+\Theta_{2}(x;\beta)\,|x|^{2}\,\mathbf{b}\cdot\mathbf{\Phi}_{2}(x),
&w_{3}(\cdot;P)&:=w_{2}+\tau\,\mathcal{L}(x;\beta).
\end{aligned}
\end{equation}
Then
\begin{equation}\label{coupled}
\begin{aligned}
\delta_{\mathrm{S}}(u) \geq \inf_{P}\Bigg[&
\frac{D_{1}}{2}\left\|\nabla(u-w_{0})\right\|_{2}^{2}
+\frac{D_{2}-D_{1}}{2}\left\|\nabla(u-w_{1})\right\|_{2}^{2}\\
&+\frac{4-D_{2}}{2}\left\|\nabla(u-w_{2})\right\|_{2}^{2}
+\frac{D_{3}-4}{2}\left\|\nabla(u-w_{3})\right\|_{2}^{2}\Bigg],
\end{aligned}
\end{equation}
and, if $u$ is real-valued, the infimum can equivalently be taken over the real
parameters.
\end{theorem}

Let $u$ be real-valued, so that the infimum in \eqref{coupled} may be taken over the
real parameters. Keeping only the first term and noting that the profiles
$w_{0}=\alpha e^{-\beta|x|^{2}}$ are exactly the elements of $E_{\mathrm{SHUP}}$, we
recover \eqref{DoLLstab} with its sharp constant $\frac{D_{1}}{2}$. The remaining three
terms are new, and they give the stability of \eqref{DoLLstab} itself. Indeed, for every
$P=(\alpha,\tau,\mathbf{a},\mathbf{b},\beta)$ we have
$w_{1}(\cdot;P)=w_{3}(\cdot;P_{1})$ and $w_{2}(\cdot;P)=w_{3}(\cdot;P_{2})$, where
$P_{1}$ and $P_{2}$ are obtained from $P$ by setting $\tau=0$, and also $\mathbf{b}=0$
in the case of $P_{1}$. Hence each of the last three terms of \eqref{coupled} is bounded
from below by $\inf_{P'}\|\nabla(u-w_{3}(\cdot;P'))\|_{2}^{2}$, and, since
$(D_{2}-D_{1})+(4-D_{2})+(D_{3}-4)=D_{3}-D_{1}$, we obtain
\[
\delta_{\mathrm{S}}(u)
-\frac{D_{1}}{2}\inf_{u^{*}\in E_{\mathrm{SHUP}}}
\left\|\nabla(u-u^{*})\right\|_{2}^{2}
\ \geq\ \frac{D_{3}-D_{1}}{2}\,
\inf_{P}\left\|\nabla\left(u-w_{3}(\cdot;P)\right)\right\|_{2}^{2} .
\]
Each of the four constants is optimal on the spherical harmonic node that governs it;
see Subsection \ref{Schains}, and also Remark \ref{Rspectral}, where the third of
them, $\frac{4-D_{2}}{2}$, is the one corresponding to the index $j=m$.

In $w_{3}$, the profile $\mathcal{L}$ cannot be replaced by the naive radial part
$(\alpha_{1}+\alpha_{2}|x|^{2})e^{-\beta|x|^{2}}$ with $\alpha_{1}$ coupled to the
Gaussian coefficient of $w_{0}$: for each fixed $\beta$ the two radial families span the
same two dimensional space, but with that parametrization \eqref{coupled} is no longer
true; see Remark \ref{Rnaive}.

As $N \to \infty$, one has $D_{1}=2-\frac{4}{N}+O(N^{-2})$,
$D_{2}=4-\frac{8}{N}+O(N^{-2})$ and $D_{3}=6-\frac{12}{N}+O(N^{-2})$, so that the four
constants in \eqref{coupled} tend to $1$, $1$, $0$ and $1$; in particular, the finest
remainder term carries asymptotically the same weight as the first one.

Since a curl-free vector field is a gradient, Theorem \ref{Tcascade} gives us at once
an improved stability estimate for the HUP with curl-free vector fields
\eqref{HUPcurl}. Indeed, let
\[
\delta_{\mathrm{CF}}(\mathbf{U}):=
\left(\int_{\mathbb{R}^{N}} |\nabla \mathbf{U}|^{2}\,dx\right)^{\frac{1}{2}}
\left(\int_{\mathbb{R}^{N}} |x|^{2} |\mathbf{U}|^{2}\,dx\right)^{\frac{1}{2}}
-\frac{N+2}{2}\int_{\mathbb{R}^{N}} |\mathbf{U}|^{2}\,dx,
\]
and let $E_{\mathrm{CF}}:=\left\{\alpha e^{-\beta|x|^{2}}x
:\alpha \in \mathbb{R},\ \beta>0\right\}$.

\begin{corollary}\label{Ccurlfree}
Let $N \geq 2$ and let $\mathbf{U}=\nabla u$ be a curl-free vector field with
$u \in X$, possibly complex-valued. Then, with $w_{0},\ldots,w_{3}$ as in \eqref{wj},
\begin{equation}\label{curlcoupled}
\begin{aligned}
\delta_{\mathrm{CF}}(\mathbf{U}) \geq \inf_{P}\Bigg[&
\frac{D_{1}}{2}\left\|\mathbf{U}-\nabla w_{0}\right\|_{2}^{2}
+\frac{D_{2}-D_{1}}{2}\left\|\mathbf{U}-\nabla w_{1}\right\|_{2}^{2}\\
&+\frac{4-D_{2}}{2}\left\|\mathbf{U}-\nabla w_{2}\right\|_{2}^{2}
+\frac{D_{3}-4}{2}\left\|\mathbf{U}-\nabla w_{3}\right\|_{2}^{2}\Bigg],
\end{aligned}
\end{equation}
and, if $\mathbf{U}$ is real-valued, the infimum can equivalently be taken over the real
parameters. In particular, keeping only the first term and noting that
$\nabla\left(\alpha e^{-\beta|x|^{2}}\right)=-2\alpha\beta\,e^{-\beta|x|^{2}}x$ runs
over $E_{\mathrm{CF}}$, we obtain, for real-valued $\mathbf{U}$,
\[
\delta_{\mathrm{CF}}(\mathbf{U}) \geq
\frac{D_{1}}{2}\inf_{\mathbf{U}^{*}\in E_{\mathrm{CF}}}
\|\mathbf{U}-\mathbf{U}^{*}\|_{2}^{2},
\]
and this constant is sharp.
\end{corollary}

\subsection{Comparison with the literature}\label{Scomparison}

Uncertainty principles in which the order of the derivative is fractional have been
considered before. In this subsection we describe the results that are closest to ours,
and we explain precisely what is new here.

On the one hand, Steinerberger \cite{Ste20} and then Xiao \cite{Xia22} studied an
$L^{p}$-uncertainty principle for the positively ordered pair
$\left\{(-\Delta)^{\alpha/2},(-\Delta)^{\beta/2}\right\}$. Their approach is based on
Fourier analysis and on the fractional Schr\"odinger equation, and their result reads
as follows:
\begin{equation}\label{Xiao}
\|f\|_{L^{p}}^{\alpha+\beta} \leq \kappa\,
\left\||\cdot|^{\alpha}f\right\|_{L^{p}}^{\beta}
\left\||\cdot|^{\beta}\widehat{f}\right\|_{L^{p'}}^{\alpha},
\end{equation}
which was later extended to the Lorentz spaces $L^{p,q}(\mathbb{R}^{N})$ by Fu and Xiao
\cite{FX23}. On the other hand, Kumar, Ponce Vanegas and Vega \cite{KPV22} gave a new
proof of the fractional uncertainty principle of Hirschman,
\begin{equation}\label{KPV}
\left\||x|^{\delta}f\right\|_{L^{2}}
\left\||\xi|^{\delta}\widehat{f}\right\|_{L^{2}}
\ \geq\ a_{\delta}^{2}\,\|f\|_{L^{2}}^{2},
\qquad 0<\delta<1,
\end{equation}
and studied the corresponding dynamical version along the Schr\"odinger flow.

In all of these results the constants are not explicit. In \eqref{Xiao} the constant
$\kappa$ depends on $\{\alpha,\beta,N,p\}$ and is obtained through a normalization
argument, so that neither its optimal value nor the corresponding optimizers are known.
In \eqref{KPV} the infimum is attained, and the minimizer $Q_{\delta}$ is shown to
satisfy $Q_{\delta}(x)\approx|x|^{-N-4\delta}$ at infinity, but the value of
$a_{\delta}$ is not known; determining it is in fact listed as an open problem in
\cite{KPV22}. In our results, by contrast, the constants are computed exactly, the
optimizers are characterized, and we also obtain the sharp stability estimates.

The essential difference lies in the weights. In \eqref{Xiao} and \eqref{KPV}, the
function $f$ is measured against the power weights $|x|^{\alpha}$ and $|x|^{\delta}$
with arbitrary exponents, while in \eqref{frac1} and \eqref{frac2} the weight is always
the first order one, namely $|x|^{2}|u|^{2}$ and $|x|^{2}|\nabla u|^{2}$, and it is the
order of the derivative that varies. This is not a matter of presentation. Indeed, on
the Fourier side, $\||\cdot|^{\delta}f\|_{L^{2}}$ becomes the fractional order term
$\|(-\Delta_{\xi})^{\delta/2}\widehat{f}\|_{L^{2}}$, which belongs to the first order
family \eqref{CKN} only when $\delta=1$; see Remark \ref{Ronlytwolines}. The symmetric
weights of \eqref{KPV} therefore require a control of the commutator
$(-\Delta)^{\kappa/2}(x_{j}u)-x_{j}(-\Delta)^{\kappa/2}u$, which is the object of the
forthcoming paper announced in Remark \ref{Ronlytwolines}; this also suggests a
structural reason for the fact that the sharp constant $a_{\delta}$ in \eqref{KPV} is
still unknown. What all these results have in common is the classical Heisenberg
Uncertainty Principle, which corresponds to $p=2$, $\alpha=\beta=1$ in \eqref{Xiao}, to
the endpoint $\delta=1$ of \eqref{KPV}, and to $s=1$ in \eqref{frac1}.

It is worth noting that the optimizers in Theorem \ref{Tfrac1} and Theorem
\ref{Tfrac2} are Gaussians only in the two classical cases $s=1$ and $s=2$. Indeed, for
$0<\sigma<2$ the profile $G_{\sigma,\lambda}$ defined in \eqref{profile} is, up to a
normalization, the symmetric $\sigma$-stable density, and it decays like
$|x|^{-N-\sigma}$ at infinity. In this respect our optimizers behave like the minimizer
$Q_{\delta}$ of \cite{KPV22} recalled above: the fractional order
destroys the Gaussian decay. The difference is that here the profiles are known
explicitly, on the Fourier side, for every $s>0$.

\subsection*{Organization of the paper}

Our paper is organized as follows. In Section \ref{Sprelim}, we collect the Fourier
dictionary and the facts about spherical harmonics that will be used later. Section \ref{Sproofs} opens with the few-line proof of Theorem \ref{TC} announced in
Subsection \ref{Spov}, and then proves Theorems \ref{Tfrac1}, \ref{Tfrac2} and
\ref{Tfracneg}.
Section \ref{Sfraccascade} contains the two building blocks and the proofs of Theorems
\ref{Tfraccascade} and \ref{Tfrac2cascade}, together with Corollary
\ref{CHUPcascade}. Finally, Section \ref{Spoincare} is devoted to the second order
Heisenberg Uncertainty Principle: we establish there the two chains of weighted
Poincar\'e inequalities on the Fourier side, namely Theorem \ref{TimprovedPoincare} and
its coupled form, Theorem \ref{TcoupledPoincare}, we carry out the inverse Fourier
computations, and we prove Theorem \ref{Tcascade} and Corollary \ref{Ccurlfree}.

\section{Preliminaries}\label{Sprelim}

\subsection{Fourier identities and the dictionary}\label{Sdictionary}

Throughout this paper, we use the normalization
\[
\mathcal{F}(f)(\xi)=\widehat{f}(\xi)
=\int_{\mathbb{R}^{N}} f(x)e^{-2\pi i x\cdot\xi}\,dx,
\qquad
\mathcal{F}^{-1}(g)(x)
=\int_{\mathbb{R}^{N}} g(\xi)e^{2\pi i x\cdot\xi}\,d\xi,
\]
so that
$\mathcal{F}(\partial_{x_{k}}f)(\xi)=2\pi i\,\xi_{k}\widehat{f}(\xi)$ and
$\mathcal{F}(x_{h}f)(\xi)=-\frac{1}{2\pi i}\partial_{\xi_{h}}\widehat{f}(\xi)$.

The following lemma is the dictionary that translates the three integrals of
\eqref{CKN}, written for $\widehat{u}$, into classical quantities in the physical
variable.

\begin{lemma}\label{Ldictionary}
For $u \in C_{0}^{\infty}(\mathbb{R}^{N})$ and $\kappa \in \mathbb{R}$, one has
\begin{align}
\int_{\mathbb{R}^{N}}|\xi|^{2\kappa}|\widehat{u}(\xi)|^{2}\,d\xi
&=(2\pi)^{-2\kappa}\int_{\mathbb{R}^{N}}
\left|(-\Delta)^{\kappa/2}u\right|^{2}dx,
\label{dictA}\\
\int_{\mathbb{R}^{N}}|\nabla\widehat{u}(\xi)|^{2}\,d\xi
&=(2\pi)^{2}\int_{\mathbb{R}^{N}}|x|^{2}|u|^{2}\,dx,
\label{dictB}\\
\int_{\mathbb{R}^{N}}|\xi|^{2}|\nabla\widehat{u}(\xi)|^{2}\,d\xi
&=\int_{\mathbb{R}^{N}}|x|^{2}|\nabla u|^{2}\,dx.
\label{dictC}
\end{align}
\end{lemma}

\begin{proof}
The identity \eqref{dictA} is just the definition of $(-\Delta)^{\kappa/2}$ together
with the Plancherel theorem. For \eqref{dictB}, we use
$\mathcal{F}(x_{j}u)(\xi)=-\frac{1}{2\pi i}\partial_{\xi_{j}}\widehat{u}(\xi)$ and the
Plancherel theorem once more:
\[
\int_{\mathbb{R}^{N}}|x|^{2}|u|^{2}\,dx
=\sum_{j=1}^{N}\int_{\mathbb{R}^{N}}|x_{j}u|^{2}\,dx
=\sum_{j=1}^{N}\int_{\mathbb{R}^{N}}
\left|\mathcal{F}(x_{j}u)(\xi)\right|^{2}d\xi
=\frac{1}{(2\pi)^{2}}\int_{\mathbb{R}^{N}}|\nabla\widehat{u}|^{2}\,d\xi.
\]
For \eqref{dictC}, we write $u_{k}=\partial u/\partial x_{k}$ and we use the Plancherel
theorem again:
\[
\int_{\mathbb{R}^{N}}|x|^{2}|\nabla u|^{2}\,dx
=\sum_{k,h=1}^{N}\int_{\mathbb{R}^{N}}|x_{h}u_{k}|^{2}\,dx
=\sum_{k,h=1}^{N}\int_{\mathbb{R}^{N}}
\left|\mathcal{F}(x_{h}u_{k})(\xi)\right|^{2}d\xi.
\]
Since
\[
\mathcal{F}(x_{h}u_{k})(\xi)
=-\frac{1}{2\pi i}\partial_{\xi_{h}}\left[\mathcal{F}(u_{k})\right](\xi)
=-\partial_{\xi_{h}}\left[\xi_{k}\widehat{u}\right](\xi)
=-\delta_{kh}\widehat{u}-\xi_{k}\partial_{\xi_{h}}\widehat{u},
\]
we derive that
\[
\int_{\mathbb{R}^{N}}|x|^{2}|\nabla u|^{2}\,dx
=N\int_{\mathbb{R}^{N}}|\widehat{u}|^{2}\,d\xi
+\int_{\mathbb{R}^{N}}|\xi|^{2}|\nabla\widehat{u}|^{2}\,d\xi
+2\sum_{h=1}^{N}\int_{\mathbb{R}^{N}}\xi_{h}\,
\mathrm{Re}\left[\partial_{\xi_{h}}\widehat{u}\cdot
\overline{\widehat{u}}\right]d\xi.
\]
By an integration by parts, the last sum is equal to
$\sum_{h}\int \xi_{h}\partial_{\xi_{h}}|\widehat{u}|^{2}\,d\xi
=-N\int|\widehat{u}|^{2}\,d\xi$, and therefore \eqref{dictC} follows.
\end{proof}

We note that \eqref{dictA} with $\kappa=2$ and $\kappa=1$ gives the two familiar
identities
\[
\int_{\mathbb{R}^{N}}|\Delta u|^{2}dx
=(2\pi)^{4}\int_{\mathbb{R}^{N}}|\xi|^{4}|\widehat{u}|^{2}d\xi,
\qquad
\int_{\mathbb{R}^{N}}|\nabla u|^{2}dx
=(2\pi)^{2}\int_{\mathbb{R}^{N}}|\xi|^{2}|\widehat{u}|^{2}d\xi.
\]
The identity \eqref{dictC} is the key point of our approach. Indeed, it says that the
weight $|x|^{2}$ acting on $\nabla u$ becomes on the Fourier side again a first order
quantity, namely $|\xi|^{2}|\nabla\widehat{u}|^{2}$, and that no derivative is lost in
this process.

\begin{remark}\label{Ronlytwolines}
The first factor of \eqref{CKN}, written for $\widehat{u}$, is
$\int_{\mathbb{R}^{N}}|\nabla\widehat{u}|^{2}|\xi|^{-2b}\,d\xi$. Writing
$\mathcal{F}(x_{j}u)=-\frac{1}{2\pi i}\partial_{\xi_{j}}\widehat{u}$ as above, one has,
for every $\kappa \in \mathbb{R}$,
\[
\int_{\mathbb{R}^{N}}|\xi|^{2\kappa}|\nabla\widehat{u}(\xi)|^{2}\,d\xi
=\frac{1}{(2\pi)^{2\kappa-2}}\sum_{j=1}^{N}\int_{\mathbb{R}^{N}}
\left|(-\Delta)^{\kappa/2}(x_{j}u)\right|^{2}dx,
\]
and therefore the choices $\kappa=0$ and $\kappa=1$, that is, $b=0$ and $b=-1$, are the
only ones for which no fractional power of the Laplacian is left on the right-hand
side. This is precisely why we restrict ourselves to these two lines. A systematic
study of the general case, which requires an analysis of the commutator
$(-\Delta)^{\kappa/2}(x_{j}u)-x_{j}(-\Delta)^{\kappa/2}u$, will be carried out in a
forthcoming paper. Let us also mention that, for a general $\gamma>0$,
\[
\int_{\mathbb{R}^{N}}|x|^{2\gamma}|u|^{2}\,dx
=(2\pi)^{-2\gamma}\int_{\mathbb{R}^{N}}
\left|(-\Delta_{\xi})^{\gamma/2}\widehat{u}\right|^{2}d\xi,
\]
so that a weight $|x|^{2\gamma}|u|^{2}$ with $\gamma \neq 1$ corresponds on the Fourier
side to a term of fractional order $\gamma$, which does not belong to the first order
family \eqref{CKN}. For $\gamma=1$ we recover \eqref{dictB}, since
$\int_{\mathbb{R}^{N}}|(-\Delta_{\xi})^{1/2}\widehat{u}|^{2}d\xi
=\int_{\mathbb{R}^{N}}|\nabla\widehat{u}|^{2}d\xi$.
\end{remark}

Before we proceed, let us make one remark on the function spaces. The inequality
\eqref{CKN} is stated for $v \in C_{0}^{\infty}(\mathbb{R}^{N}\setminus\{0\})$, while
we shall apply it to $v=\widehat{u}$, which is neither compactly supported nor
vanishing at the origin. This causes no difficulty. Indeed, both sides of \eqref{CKN}
are continuous with respect to the norm
$\left(\int_{\mathbb{R}^{N}}|\nabla v|^{2}|x|^{-2b}dx
+\int_{\mathbb{R}^{N}}|v|^{2}|x|^{-2a}dx\right)^{1/2}$. Moreover, in our two
applications we have $b \in \{0,-1\}$, so that $b \leq \frac{N-2}{2}$ for all the
dimensions considered here. As we will explain below, this is precisely the condition
under which the origin has zero capacity with respect to the Dirichlet form
$\int_{\mathbb{R}^{N}}|\nabla v|^{2}|x|^{-2b}dx$, and therefore we can find cut-off
functions $\eta_{\varepsilon}$, vanishing near the origin and equal to $1$ outside a
small ball, such that
$\int_{\mathbb{R}^{N}}|\nabla\eta_{\varepsilon}|^{2}|x|^{-2b}dx \to 0$. Hence
$C_{0}^{\infty}(\mathbb{R}^{N}\setminus\{0\})$ is dense in the space of the functions
for which the above norm is finite, and consequently \eqref{CKN}, as well as Theorem
\ref{TB}, remains valid for every such $v$. For
$u \in C_{0}^{\infty}(\mathbb{R}^{N})$ and for the parameters used below, this is
always the case: the function $\widehat{u}$ is smooth and decays, together with all its
derivatives, faster than any power of $|\xi|$, and near the origin the weights
$|\xi|^{2s}$, $|\xi|^{2}$, $|\xi|^{s}$ and $|\xi|^{s-1}$ are locally integrable. Indeed, in Theorem
\ref{Tfrac1} we have $2s>-2\geq-N$ and $s-1>-2\geq-N$ since $s>-1$ and $N \geq 2$,
while in Theorem \ref{Tfracneg} the assumption $s>-\frac{N}{2}$ gives $2s>-N$, together
with $s>-N$ and $s-1>-N$.

It is worth adding one comment on the condition $b \leq \frac{N-2}{2}$, which appears in
Theorem \ref{TA} and Theorem \ref{TB} and which separates the regions
$\mathcal{A}_{1},\mathcal{B}_{1}$ from $\mathcal{A}_{2},\mathcal{B}_{2}$. This
condition is exactly the condition under which the origin has zero capacity with respect
to the Dirichlet form $\int_{\mathbb{R}^{N}}|\nabla v|^{2}|x|^{-2b}dx$. Indeed, for the
radial cut-off functions the corresponding capacity is controlled by
\[
\left(\int_{\varepsilon}^{1}r^{2b-N+1}\,dr\right)^{-1},
\]
and the integral diverges as $\varepsilon \to 0$ if and only if $2b-N+1 \leq -1$, that
is, if and only if $b \leq \frac{N-2}{2}$. This explains why the above density argument works
precisely in the regions that we use. In the complementary regions
$\mathcal{A}_{2},\mathcal{B}_{2}$, where $b \geq \frac{N-2}{2}$, the origin is seen by
the Dirichlet form, and the functions in the completion of
$C_{0}^{\infty}(\mathbb{R}^{N}\setminus\{0\})$ vanish there; this is consistent with the
fact that in those regions all the optimizers of Theorem \ref{TA} vanish at the origin
as well. Since the Fourier
dictionary forces $b \in \{0,-1\}$, we have $b \leq \frac{N-2}{2}$ in all the cases
treated in this paper, with the only exception of $b=0$ and $N=1$; we note in passing
that for $b=0$ and $N=2$ the two regions meet, and that the two branches of the maximum
in Theorem \ref{TA} then give the same constant $\frac{1+s}{2}$, in agreement with
$2(b+1)-N=0$.

\subsection{Spherical harmonics}\label{Ssph}

Let $\{\phi_{k,i}\}_{i=1}^{N_{k}}$, $k=0,1,2,\ldots$, be an orthonormal basis of
$L^{2}(\mathbb{S}^{N-1})$ consisting of spherical harmonics, with
\[
-\Delta_{\mathbb{S}^{N-1}}\phi_{k,i}=c_{k}\phi_{k,i},\qquad c_{k}=k(N+k-2),\qquad
N_{k}=\binom{N+k-1}{k}-\binom{N+k-3}{k-2}.
\]
In particular, $N_{0}=1$ and $\phi_{0,1}\equiv|\mathbb{S}^{N-1}|^{-1/2}$, while
$N_{1}=N$ and the functions $\phi_{1,i}$ are multiples of the coordinate functions
$\theta_{i}$.
For $v \in C_{0}^{\infty}(\mathbb{R}^{N})$, we decompose
\begin{equation}\label{sphdecomp}
v(x)=v(r\theta)=\sum_{k=0}^{\infty}\sum_{i=1}^{N_{k}}v_{k,i}(r)\phi_{k,i}(\theta),
\qquad r=|x|,\quad \theta=\frac{x}{|x|},
\end{equation}
so that, for any weight $w(r)>0$,
\begin{equation}\label{sphgrad}
\int_{\mathbb{R}^{N}}|\nabla v|^{2}w(|x|)\,dx
=\sum_{k=0}^{\infty}\sum_{i=1}^{N_{k}}\int_{0}^{\infty}
\left(|v_{k,i}'|^{2}+c_{k}\frac{|v_{k,i}|^{2}}{r^{2}}\right)w(r)r^{N-1}\,dr,
\end{equation}
\[
\int_{\mathbb{R}^{N}}|v|^{2}w(|x|)\,dx
=\sum_{k=0}^{\infty}\sum_{i=1}^{N_{k}}\int_{0}^{\infty}
|v_{k,i}|^{2}w(r)r^{N-1}\,dr.
\]

\begin{lemma}\label{Lordering}
Let $N \geq 2$ and let $D_{k}$ be as in \eqref{Dk}. Then \eqref{ordering} holds, that
is, $0<D_{1}<D_{2}<4<D_{3}$ and $D_{3}-4<D_{1}$. Moreover,
\begin{equation}\label{ordering2}
D_{1}<D_{2}-D_{1}
\qquad\text{and}\qquad
D_{4}-D_{1}>4.
\end{equation}
\end{lemma}

\begin{proof}
Since $c_{1}<c_{2}<c_{3}$, the map $k \mapsto D_{k}$ is strictly increasing, and
therefore $0<D_{1}<D_{2}<D_{3}$. Next, $D_{2}<4$ is equivalent to
$N^{2}+8N<(N+4)^{2}$, that is, to $0<16$, which is always true. Also, $D_{3}>4$ is
equivalent to $N^{2}+12N+12>(N+4)^{2}=N^{2}+8N+16$, that is, to $4N>4$, which holds
since $N \geq 2$. Finally, $D_{3}-4<D_{1}$ means
$\sqrt{N^{2}+12N+12}<4+\sqrt{N^{2}+4N-4}$. Squaring both sides, this is equivalent to
$8N<8\sqrt{N^{2}+4N-4}$, that is, to $N^{2}<N^{2}+4N-4$, and again this is exactly
$N>1$.

It remains to prove \eqref{ordering2}. The first inequality, $D_{2}>2D_{1}$, means
$\sqrt{N^{2}+8N}+N>2\sqrt{N^{2}+4N-4}$. Squaring both sides, it is equivalent to
$N\sqrt{N^{2}+8N}>N^{2}+4N-8$, and squaring once more, which is legitimate since both
sides are positive for $N \geq 2$, it becomes
$N^{4}+8N^{3}>N^{4}+8N^{3}-64N+64$, that is, $N>1$. For the second inequality, since
$c_{4}=4(N+2)$, we have $D_{4}=\sqrt{N^{2}+16N+32}-N$, and $D_{4}-D_{1}>4$ means
$\sqrt{N^{2}+16N+32}>4+\sqrt{N^{2}+4N-4}$. Squaring both sides, this is equivalent to
$12N+20>8\sqrt{N^{2}+4N-4}$, that is, to $(3N+5)^{2}>4\left(N^{2}+4N-4\right)$, and
hence to $5N^{2}+14N+41>0$, which is always true.
\end{proof}

Finally, we state the one dimensional Poincar\'e inequality with Gaussian type measure
from \cite{DLLN26}, in the form that we will use. In \cite{DLLN26} it is stated for the
measure $e^{-\rho^{2}/2}\rho^{\alpha-1}\,d\rho$, and the change of variable
$\rho=\sqrt{2}\,r$
gives the following version.

We will use the following one dimensional Poincar\'e inequality, which is the case
$\sigma=2$ of Lemma \ref{Llaguerre} below.

\begin{theorema}\label{TD}
Let $\alpha>0$. For every
$\psi \in W^{1,2}\left((0,\infty),e^{-r^{2}}r^{\alpha-1}\,dr\right)$, one has
\[
\int_{0}^{\infty}|\psi'|^{2}e^{-r^{2}}r^{\alpha-1}\,dr
\geq 4\inf_{c \in \mathbb{R}}\int_{0}^{\infty}
|\psi-c|^{2}e^{-r^{2}}r^{\alpha-1}\,dr.
\]
Here the constant $4$ is sharp and is attained exactly by the functions
$\psi(r)=a_{0}+a_{1}r^{2}$. Moreover, one also has
\[
\int_{0}^{\infty}|\psi'|^{2}e^{-r^{2}}r^{\alpha-1}\,dr
-4\inf_{c \in \mathbb{R}}\int_{0}^{\infty}|\psi-c|^{2}e^{-r^{2}}r^{\alpha-1}\,dr
\geq 4\inf_{c,d \in \mathbb{R}}\int_{0}^{\infty}
|\psi-c-dr^{2}|^{2}e^{-r^{2}}r^{\alpha-1}\,dr.
\]
\end{theorema}

%---------------------------------------------------------------------------------------------
%---------------------------------------------------------------------------------------------
\section[Proofs of the main results]{A short proof of Theorem \ref{TC}, and the proofs of Theorems \ref{Tfrac1}, \ref{Tfrac2} and \ref{Tfracneg}}
\label{Sproofs}

We begin with the proof of Theorem \ref{TC} announced in Subsection \ref{Spov}, which
displays the mechanism of the paper in its simplest form.

\begin{proof}[Proof of Theorem \ref{TC}]
By \eqref{dict1} with $\kappa=2$ and $\kappa=1$, and by the second identity in
\eqref{dict2}, we have
\[
\begin{aligned}
\delta_{\mathrm{S}}(u)
&=(2\pi)^{2}\bigg[
\left(\int_{\mathbb{R}^{N}}|\xi|^{4}|\widehat{u}|^{2}d\xi\right)^{\frac12}
\left(\int_{\mathbb{R}^{N}}|\xi|^{2}|\nabla\widehat{u}|^{2}d\xi\right)^{\frac12}\\
&\qquad-\frac{N+2}{2}\int_{\mathbb{R}^{N}}|\xi|^{2}|\widehat{u}|^{2}d\xi\bigg]
=(2\pi)^{2}\,\delta_{2}(\widehat{u}),
\end{aligned}
\]
which is \eqref{key}, where $\delta_{2}$ is the CKN deficit \eqref{L2CKNdeficit} for
$(a,b)=(-2,-1)$, since
$|\xi|^{-2a}=|\xi|^{4}$, $|\xi|^{-2b}=|\xi|^{2}$, $|\xi|^{-(a+b+1)}=|\xi|^{2}$ and
$C(N,-2,-1)=\frac{N+2}{2}$ by Theorem \ref{TA}. The nonnegativity of $\delta_{2}$ is the
sharp inequality \eqref{2HUP}. Moreover $(-2,-1)\in\mathcal{A}_{1}$ and
$C_{\mathrm{PI}}(-2,-1,N)=\min\{4,\sqrt{N^{2}+4N-4}-N\}=D_{1}$ by Lemma \ref{Lordering},
so that the first case of Theorem \ref{TB} gives
\[
\delta_{2}(\widehat{u})\geq\frac{D_{1}}{2}
\inf_{c \in \mathbb{C},\,\lambda>0}\int_{\mathbb{R}^{N}}
\left|\widehat{u}-c\,e^{-\frac{\lambda|\xi|^{2}}{2}}\right|^{2}|\xi|^{2}\,d\xi .
\]
Finally, $\mathcal{F}^{-1}\left(e^{-\lambda|\xi|^{2}/2}\right)
=\left(\frac{2\pi}{\lambda}\right)^{N/2}e^{-\frac{2\pi^{2}}{\lambda}|x|^{2}}$ is a
Gaussian, and \eqref{dict1} with $\kappa=1$ turns the weighted $L^{2}$ distance on the
right-hand side into
$\frac{1}{(2\pi)^{2}}\int_{\mathbb{R}^{N}}|\nabla(u-c\,G_{2,\lambda})|^{2}dx$. Since
$u$ is real-valued and $G_{2,\lambda}$ is real, writing $c=c_{1}+ic_{2}$ gives
$\int|\nabla(u-cG_{2,\lambda})|^{2}
=\int|\nabla(u-c_{1}G_{2,\lambda})|^{2}+c_{2}^{2}\int|\nabla G_{2,\lambda}|^{2}$, so
that the infimum over $c \in \mathbb{C}$ is attained at a real $c$, and the competitors
$c\,G_{2,\lambda}$ are then exactly the elements of $E_{\mathrm{SHUP}}$.
Multiplying by $(2\pi)^{2}$ gives \eqref{DoLLstab}, and the sharpness of the constant
follows from the sharpness part of Theorem \ref{TB}.
\end{proof}

The proofs of the fractional theorems follow the same three steps: translate, apply
Theorem \ref{TA} and Theorem \ref{TB} to $\widehat{u}$, and invert the Fourier
transform. We are now ready to prove Theorems \ref{Tfrac1}, \ref{Tfrac2} and \ref{Tfracneg}.

\begin{proof}[Proof of Theorem \ref{Tfrac1}]
We apply \eqref{CKN} to $v=\widehat{u}$ with the parameters
\[
a=-s,\qquad b=0.
\]
Then $b+1-a=1+s>0$ since $s>-1$, and $b=0\leq\frac{N-2}{2}$ since $N \geq 2$, and hence
$(-s,0)\in\mathcal{A}_{1}$. Therefore Theorem \ref{TA} gives us
\[
C(N,-s,0)=\frac{|N-(a+b+1)|}{2}=\frac{|N-(1-s)|}{2}=\frac{N+s-1}{2},
\]
where we used $N-1+s>0$, which holds since $s>-1$ and $N \geq 2$.
For these parameters, $|\xi|^{-2a}=|\xi|^{2s}$, $|\xi|^{-2b}=1$ and
$|\xi|^{-(a+b+1)}=|\xi|^{s-1}$, so that, by \eqref{dictA},
\[
\begin{aligned}
\int_{\mathbb{R}^{N}}\frac{|\widehat{u}|^{2}}{|\xi|^{2a}}\,d\xi
&=(2\pi)^{-2s}\int_{\mathbb{R}^{N}}\left|(-\Delta)^{s/2}u\right|^{2}dx,\\
\int_{\mathbb{R}^{N}}\frac{|\widehat{u}|^{2}}{|\xi|^{a+b+1}}\,d\xi
&=(2\pi)^{1-s}\int_{\mathbb{R}^{N}}
\left|(-\Delta)^{\frac{s-1}{4}}u\right|^{2}dx,
\end{aligned}
\]
while $\int_{\mathbb{R}^{N}}|\nabla\widehat{u}|^{2}\,d\xi
=(2\pi)^{2}\int_{\mathbb{R}^{N}}|x|^{2}|u|^{2}\,dx$ by \eqref{dictB}. Substituting
these three identities into \eqref{CKN}, we obtain
\[
\begin{aligned}
&(2\pi)^{1-s}\left(\int_{\mathbb{R}^{N}}
\left|(-\Delta)^{s/2}u\right|^{2}dx\right)^{\frac{1}{2}}
\left(\int_{\mathbb{R}^{N}}|x|^{2}|u|^{2}\,dx\right)^{\frac{1}{2}}\\
&\qquad\geq \frac{N+s-1}{2}(2\pi)^{1-s}\int_{\mathbb{R}^{N}}
\left|(-\Delta)^{\frac{s-1}{4}}u\right|^{2}dx,
\end{aligned}
\]
which is exactly \eqref{frac1}, since the powers of $2\pi$ cancel each other. Moreover,
by Theorem \ref{TA}, the equality holds if and only if
$\widehat{u}(\xi)=D\exp\left(t|\xi|^{1+s}/(1+s)\right)$ with $t<0$, that is, if and
only if $u$ is a multiple of some profile $G_{1+s,\lambda}$ defined in
\eqref{profile}.

For the stability statement, we observe the elementary but useful identity
\eqref{magic}, which gives
\[
\begin{aligned}
\sqrt{(N-2-2b)^{2}+4(N-1)}-(N-2-2b)\Big|_{b=0}
&=\sqrt{(N-2)^{2}+4(N-1)}-(N-2)\\
&=N-(N-2)=2.
\end{aligned}
\]
Hence, for every $s>-1$ and every $N \geq 2$,
\[
C_{\mathrm{PI}}(-s,0,N)=\min\left\{2(1+s),\,2\right\}=2\min\left\{1+s,\,1\right\},
\]
which is equal to $2$ when $s \geq 0$, and to $2(1+s)$ when $-1<s<0$. Therefore the
first case of Theorem \ref{TB} with $(a,b)=(-s,0)$ gives us
\[
\delta_{2}(\widehat{u}) \geq
\min\left\{1+s,\,1\right\}
\inf_{c \in \mathbb{C},\, \lambda>0}\int_{\mathbb{R}^{N}}
\left|\widehat{u}-c\,e^{-\frac{\lambda|\xi|^{1+s}}{1+s}}\right|^{2}
|\xi|^{s-1}\,d\xi.
\]
It remains to undo the Fourier transform. By \eqref{dictA} with
$2\kappa=s-1$, applied to the function
$g:=\mathcal{F}^{-1}\left(\widehat{u}-c\,e^{-\lambda|\xi|^{1+s}/(1+s)}\right)$, and by
$\mathcal{F}^{-1}\left(e^{-\lambda|\xi|^{1+s}/(1+s)}\right)=G_{1+s,\lambda}$, we get
\[
\int_{\mathbb{R}^{N}}
\left|\widehat{u}-c\,e^{-\frac{\lambda|\xi|^{1+s}}{1+s}}\right|^{2}|\xi|^{s-1}\,d\xi
=(2\pi)^{1-s}\int_{\mathbb{R}^{N}}\left|(-\Delta)^{\frac{s-1}{4}}
\left(u-c\,G_{1+s,\lambda}\right)\right|^{2}dx,
\]
and the powers of $2\pi$ cancel once more. This gives \eqref{frac1stab}, and the
sharpness of the constant $\min\{1+s,1\}$ follows from the sharpness part of Theorem
\ref{TB}.
\end{proof}

\begin{proof}[Proof of Theorem \ref{Tfrac2}]
This time we apply \eqref{CKN} to $v=\widehat{u}$ with the parameters
\[
a=-s<0,\qquad b=-1.
\]
Then $b+1-a=s>0$ and $b=-1\leq\frac{N-2}{2}$, and hence
$(-s,-1)\in\mathcal{A}_{1}$, so that Theorem \ref{TA} gives us
\[
C(N,-s,-1)=\frac{|N-(a+b+1)|}{2}=\frac{N+s}{2}.
\]
For these parameters, $|\xi|^{-2a}=|\xi|^{2s}$, $|\xi|^{-2b}=|\xi|^{2}$ and
$|\xi|^{-(a+b+1)}=|\xi|^{s}$, so that, by \eqref{dictA},
\[
\begin{aligned}
\int_{\mathbb{R}^{N}}\frac{|\widehat{u}|^{2}}{|\xi|^{2a}}\,d\xi
&=(2\pi)^{-2s}\int_{\mathbb{R}^{N}}\left|(-\Delta)^{s/2}u\right|^{2}dx,\\
\int_{\mathbb{R}^{N}}\frac{|\widehat{u}|^{2}}{|\xi|^{a+b+1}}\,d\xi
&=(2\pi)^{-s}\int_{\mathbb{R}^{N}}\left|(-\Delta)^{s/4}u\right|^{2}dx,
\end{aligned}
\]
while $\int_{\mathbb{R}^{N}}|\nabla\widehat{u}|^{2}|\xi|^{2}\,d\xi
=\int_{\mathbb{R}^{N}}|x|^{2}|\nabla u|^{2}\,dx$ by \eqref{dictC}. Substituting these
three identities into \eqref{CKN} and cancelling the factor $(2\pi)^{-s}$, we obtain
\eqref{frac2}. As above, Theorem \ref{TA} shows that the equality holds if and only if
$u$ is a multiple of some profile $G_{s,\lambda}$.

For the stability statement, Theorem \ref{TB} with $(a,b)=(-s,-1)$ gives us
\[
\delta_{2}(\widehat{u}) \geq \frac{C_{\mathrm{PI}}(-s,-1,N)}{2}
\inf_{c \in \mathbb{C},\, \lambda>0}\int_{\mathbb{R}^{N}}
\left|\widehat{u}-c\,e^{-\frac{\lambda|\xi|^{s}}{s}}\right|^{2}|\xi|^{s}\,d\xi,
\]
where, since $N-2-2b=N$ when $b=-1$,
\[
\begin{aligned}
C_{\mathrm{PI}}(-s,-1,N)
&=\min\left\{2(1+b-a),\, \sqrt{(N-2-2b)^{2}+4(N-1)}-(N-2-2b)\right\}\\
&=\min\left\{2s,\, \sqrt{N^{2}+4N-4}-N\right\}.
\end{aligned}
\]
Undoing the Fourier transform exactly as in the proof of Theorem \ref{Tfrac1}, but now
with $2\kappa=s$, we obtain \eqref{frac2stab}.
\end{proof}

\begin{proof}[Proof of Theorem \ref{Tfracneg}]
The proof follows the same lines as the proofs of Theorem \ref{Tfrac1} and Theorem
\ref{Tfrac2}, and we only indicate the changes.

For the first part, we apply \eqref{CKN} to $v=\widehat{u}$ with $a=-s>0$ and $b=-1$.
Since $b+1-a=s<0$ and $b=-1\leq\frac{N-2}{2}$, we now have
$(-s,-1)\in\mathcal{B}_{1}$, and hence the second case of Theorem \ref{TA} gives us the
sharp constant
\[
C(N,-s,-1)=\frac{|N-(3b-a+3)|}{2}=\frac{|N-s|}{2}=\frac{N-s}{2},
\]
together with the optimizers
\[
\widehat{u}(\xi)=D\,|\xi|^{2(b+1)-N}
\exp\left(\frac{t|\xi|^{s}}{s}\right)
=D\,|\xi|^{-N}e^{-\frac{t|\xi|^{-\sigma}}{\sigma}},
\qquad t>0,\ \sigma:=-s>0,
\]
that is, by \eqref{profileB}, exactly the profiles $u=D\,G_{-s,t}^{-N}$. The dictionary
identities \eqref{dictA} and \eqref{dictC} are used exactly as in the proof of Theorem
\ref{Tfrac2}, and the powers of $2\pi$ cancel in the same way, which gives
\eqref{frac2neg}. We also note that all the integrals are finite for these profiles:
near the origin because the factor $e^{-t|\xi|^{-\sigma}/\sigma}$ decays faster than
any power of $|\xi|$, and at infinity because $|\widehat{u}(\xi)|^{2}\approx|\xi|^{-2N}$
and $s<0$.

For the stability statement, we use the second case of Theorem \ref{TB} with
$(a,b)=(-s,-1)$. Since $-a+2b+2=s$ and $N-2-2b=N$, we have
\[
\begin{aligned}
C_{\mathrm{PI}}(s,-1,N)
&=\min\left\{2(1-1-s),\,\sqrt{N^{2}+4(N-1)}-N\right\}\\
&=\min\left\{-2s,\,\sqrt{N^{2}+4N-4}-N\right\},
\end{aligned}
\]
and the competitors are
\[
c\,|\xi|^{2b+2-N}e^{\frac{\lambda|\xi|^{b+1-a}}{b+1-a}}
=c\,|\xi|^{-N}e^{\frac{\lambda|\xi|^{s}}{s}},
\]
that is, the Fourier transforms of
$c\,G_{-s,\lambda}^{-N}$. Undoing the Fourier transform with $2\kappa=s$ as before, we
obtain \eqref{frac2negstab}.

For the second part, we apply \eqref{CKN} to $v=\widehat{u}$ with $a=-s$ and $b=0$,
where now $-\frac{N}{2}<s<-1$. Since $b+1-a=1+s<0$ and $b=0\leq\frac{N-2}{2}$, we have
$(-s,0)\in\mathcal{B}_{1}$, and the second case of Theorem \ref{TA} gives us the sharp
constant
\[
C(N,-s,0)=\frac{|N-(3b-a+3)|}{2}=\frac{|N-(s+3)|}{2}=\frac{N-3-s}{2},
\]
which is positive since $s<-1$ and $N \geq 3$, together with the optimizers
\[
\widehat{u}(\xi)=D\,|\xi|^{2-N}\exp\left(\frac{t|\xi|^{1+s}}{1+s}\right)
=D\,|\xi|^{2-N}e^{-\frac{t|\xi|^{-\sigma}}{\sigma}},
\qquad t>0,\ \sigma:=-(1+s)>0,
\]
that is, the profiles $u=D\,G_{-(1+s),t}^{2-N}$. At infinity,
$|\nabla\widehat{u}(\xi)|\approx(N-2)|\xi|^{1-N}$, and hence
$\int_{\mathbb{R}^{N}}|\nabla\widehat{u}|^{2}d\xi<\infty$ precisely because $N \geq 3$.
The dictionary identities are used exactly as in the proof of Theorem \ref{Tfrac1},
which gives \eqref{frac1neg}. Finally, the second case of Theorem \ref{TB} with
$(a,b)=(-s,0)$, where now $-a+2b+2=s+2$ and, by \eqref{magic},
\[
C_{\mathrm{PI}}(s+2,0,N)=\min\left\{2\left(1-(s+2)\right),\,2\right\}
=2\min\left\{|1+s|,\,1\right\},
\]
gives \eqref{frac1negstab} after undoing the Fourier transform with $2\kappa=s-1$.
\end{proof}

Next, we compute the exact remainder term of the CKN deficit. The following identity is
the analogue, for a general Gaussian type measure, of the HUP identity established in
\cite{CFLL24}, and it is the starting point of all the stability estimates of this
paper.

\begin{lemma}\label{LCKNidentity}
Let $(a,b)\in\mathcal{A}_{1}$, write $\sigma:=1+b-a>0$, and let
$v \in C_{0}^{\infty}(\mathbb{R}^{N}\setminus\{0\})$ with $v \not\equiv 0$. Set
\[
A:=\int_{\mathbb{R}^{N}}\frac{|\nabla v|^{2}}{|x|^{2b}}dx,
\qquad
B:=\int_{\mathbb{R}^{N}}\frac{|v|^{2}}{|x|^{2a}}dx,
\qquad
\mu:=\sqrt{A/B}.
\]
Then
\begin{equation}\label{CKNidentity}
\delta_{2}(v)=\frac{1}{2\mu}\int_{\mathbb{R}^{N}}
\left|\nabla\left(v\,e^{\frac{\mu|x|^{\sigma}}{\sigma}}\right)\right|^{2}
e^{-\frac{2\mu|x|^{\sigma}}{\sigma}}\frac{dx}{|x|^{2b}}.
\end{equation}
\end{lemma}

\begin{proof}
Since $\nabla\left(e^{\frac{\mu|x|^{\sigma}}{\sigma}}\right)
=\mu|x|^{\sigma-2}x\,e^{\frac{\mu|x|^{\sigma}}{\sigma}}$, the integral on the right-hand
side of \eqref{CKNidentity} equals
\[
\int_{\mathbb{R}^{N}}\frac{\left|\nabla v+\mu|x|^{\sigma-2}xv\right|^{2}}{|x|^{2b}}dx
=A+2\mu\,\mathrm{Re}\int_{\mathbb{R}^{N}}
\frac{x\cdot\nabla v\,\overline{v}}{|x|^{2b-\sigma+2}}dx
+\mu^{2}\int_{\mathbb{R}^{N}}\frac{|v|^{2}}{|x|^{2b-2\sigma+2}}dx.
\]
Here $2b-2\sigma+2=2a$, so the last term is $\mu^{2}B$. For the middle one, an
integration by parts gives
\[
\begin{aligned}
2\,\mathrm{Re}\int_{\mathbb{R}^{N}}|x|^{\sigma-2-2b}\,x\cdot\nabla v\,\overline{v}\,dx
&=\int_{\mathbb{R}^{N}}|x|^{\sigma-2-2b}\,x\cdot\nabla\left(|v|^{2}\right)dx\\
&=-\left(N+\sigma-2-2b\right)\int_{\mathbb{R}^{N}}
\frac{|v|^{2}}{|x|^{2b-\sigma+2}}dx,
\end{aligned}
\]
and $2b-\sigma+2=a+b+1$, while $N+\sigma-2-2b=N-(a+b+1)=2C(N,a,b)$ by Theorem
\ref{TA}. Writing $M:=\int_{\mathbb{R}^{N}}|v|^{2}|x|^{-(a+b+1)}dx$, we therefore obtain
\begin{equation}\label{Qmu}
0\ \leq\ \int_{\mathbb{R}^{N}}
\frac{\left|\nabla v+\mu|x|^{\sigma-2}xv\right|^{2}}{|x|^{2b}}dx
=A-2\mu\,C(N,a,b)\,M+\mu^{2}B.
\end{equation}
Finally, $\mu^{2}B=A$ by the choice of $\mu$, so the right-hand side of \eqref{Qmu}
equals $2A-2\mu C(N,a,b)M=2\mu\left(\sqrt{AB}-C(N,a,b)M\right)=2\mu\,\delta_{2}(v)$.
\end{proof}

We now record, once and for all, the scaling argument that transfers a weighted
Poincar\'e inequality into a stability estimate for the CKN deficit. For
$(a,b)\in\mathcal{A}_{1}$ and $\sigma=1+b-a$, we write
\[
d\mu_{a,b}:=e^{-\frac{2|x|^{\sigma}}{\sigma}}\frac{dx}{|x|^{2b}},
\qquad
d\nu_{a,b}:=e^{-\frac{2|x|^{\sigma}}{\sigma}}\frac{dx}{|x|^{a+b+1}}.
\]

\begin{proposition}\label{Pfourierstab}
Let $(a,b)\in\mathcal{A}_{1}$ and let $\mathcal{K}_{0},\ldots,\mathcal{K}_{n}$ be sets
of measurable functions on $\mathbb{R}^{N}$, each of them invariant under the dilations
$F \mapsto F(t\,\cdot)$, $t>0$. Assume that, for some $\kappa_{0},\ldots,\kappa_{n}
\geq 0$ and for all $W \in C_{0}^{\infty}(\mathbb{R}^{N}\setminus\{0\})$,
\[
\int_{\mathbb{R}^{N}}|\nabla W|^{2}d\mu_{a,b}
\ \geq\ \sum_{j=0}^{n}\kappa_{j}\inf_{F \in \mathcal{K}_{j}}
\int_{\mathbb{R}^{N}}|W-F|^{2}d\nu_{a,b}.
\]
Then, for all $v \in C_{0}^{\infty}(\mathbb{R}^{N}\setminus\{0\})$,
\begin{equation}\label{fourierstab}
\delta_{2}(v)\ \geq\ \sum_{j=0}^{n}\frac{\kappa_{j}}{2}
\inf_{F \in \mathcal{K}_{j},\ \lambda>0}\int_{\mathbb{R}^{N}}
\left|v-F\,e^{-\frac{\lambda|x|^{\sigma}}{\sigma}}\right|^{2}
\frac{dx}{|x|^{a+b+1}}.
\end{equation}
\end{proposition}

\begin{proof}
Let $\mu$ be as in Lemma \ref{LCKNidentity} and set
$w:=v\,e^{\mu|x|^{\sigma}/\sigma}$, so that, by \eqref{CKNidentity},
\[
2\mu\,\delta_{2}(v)=\int_{\mathbb{R}^{N}}|\nabla w|^{2}
e^{-\frac{2\mu|x|^{\sigma}}{\sigma}}\frac{dx}{|x|^{2b}}.
\]
We now rescale. Put $t:=\mu^{1/\sigma}$, so that
$\frac{2\mu|x|^{\sigma}}{\sigma}=\frac{2|tx|^{\sigma}}{\sigma}$, and set
$W(y):=w(y/t)$. A change of variable gives
\[
\int_{\mathbb{R}^{N}}|\nabla w|^{2}e^{-\frac{2\mu|x|^{\sigma}}{\sigma}}
\frac{dx}{|x|^{2b}}
=t^{2b+2-N}\int_{\mathbb{R}^{N}}|\nabla W|^{2}d\mu_{a,b},
\]
and, for any measurable $F$,
\[
t^{2b+2-N}\int_{\mathbb{R}^{N}}|W-F|^{2}d\nu_{a,b}
=t^{2b+2-N}\,t^{\,N-(a+b+1)}\int_{\mathbb{R}^{N}}
\left|w(x)-F(tx)\right|^{2}
e^{-\frac{2\mu|x|^{\sigma}}{\sigma}}\frac{dx}{|x|^{a+b+1}},
\]
the two powers of $t$ combining into $t^{2b+2-(a+b+1)}=t^{\sigma}=\mu$. Applying the
assumed Poincar\'e inequality to $W$ and dividing by $2\mu$, we obtain
\[
\delta_{2}(v)\ \geq\ \sum_{j=0}^{n}\frac{\kappa_{j}}{2}
\inf_{F \in \mathcal{K}_{j}}\int_{\mathbb{R}^{N}}
\left|w(x)-F(tx)\right|^{2}e^{-\frac{2\mu|x|^{\sigma}}{\sigma}}
\frac{dx}{|x|^{a+b+1}}.
\]
Since each $\mathcal{K}_{j}$ is dilation invariant, $F(t\,\cdot)$ again runs over
$\mathcal{K}_{j}$; moreover
\[
\left|w(x)-F(x)\right|^{2}e^{-\frac{2\mu|x|^{\sigma}}{\sigma}}
=\left|v(x)-F(x)e^{-\frac{\mu|x|^{\sigma}}{\sigma}}\right|^{2}.
\]
Finally, the estimate can only become weaker if we take the infimum over all
$\lambda>0$ instead of the single value $\lambda=\mu$. This gives \eqref{fourierstab}.
\end{proof}

\section[The chains for the fractional families]{The chains for the fractional families: proofs of Theorems \ref{Tfraccascade} and \ref{Tfrac2cascade}}
\label{Sfraccascade}

Throughout Subsections \ref{Sblocks} and \ref{Sfourierchain} we fix $s>0$ and write
$\sigma:=1+s$, $m:=\lfloor s\rfloor+1$
and $\{s\}:=s-\lfloor s\rfloor$. We work on the Fourier side with the two measures
\[
d\mu_{s}:=e^{-\frac{2|\xi|^{\sigma}}{\sigma}}d\xi,
\qquad
d\nu_{s}:=e^{-\frac{2|\xi|^{\sigma}}{\sigma}}|\xi|^{s-1}d\xi,
\]
which are the measures attached by \cite{DLLN26} to the parameters $(a,b)=(-s,0)$.

\subsection{The two building blocks}\label{Sblocks}

The second of the two building blocks below, Lemma \ref{Llaguerre}, is a special case
of the weighted radial Poincar\'e inequalities of \cite{LLLS26}; we state
it in the normalization that we shall use, and we recall the Laguerre expansion behind
it, since the same expansion will be used again in Remark \ref{Rspectral}. The first
block, Lemma \ref{Lfacts}, is different: it factorizes each node with its own exponent,
namely $k$ on the line $b=0$ and $D_{k}/2$ on the line $b=-1$ (Lemma \ref{Lfacts2}),
whereas in \cite{LLLS26} the single exponent attached to the first
non-radial mode is used on all the non-radial nodes.

We first record the exact factorization on each non-radial node. It is the analogue,
for the line $b=0$, of the factorization used in \cite[Section 2]{DLLN26}.

\begin{lemma}\label{Lfacts}
Let $k \geq 1$ and let $\psi \in C_{0}^{\infty}((0,\infty))$ be real-valued. Write
$\psi(r)=r^{k}w(r)$. Then
\begin{equation}\label{factors}
\begin{aligned}
&\int_{0}^{\infty}\left(\left\vert\psi'(r)\right\vert^{2}+\frac{c_{k}}{r^{2}}\left\vert\psi(r)\right\vert^{2}\right)
e^{-\frac{2r^{\sigma}}{\sigma}}r^{N-1}dr\\
&\qquad=2k\int_{0}^{\infty}\left\vert\psi(r)\right\vert^{2}
e^{-\frac{2r^{\sigma}}{\sigma}}r^{N+s-2}dr
+\int_{0}^{\infty}\left\vert w'(r)\right\vert^{2}
e^{-\frac{2r^{\sigma}}{\sigma}}r^{2k+N-1}dr.
\end{aligned}
\end{equation}
\end{lemma}

\begin{proof}
Since $\psi'=kr^{k-1}w+r^{k}w'$, we have
\[
\left\vert\psi'\right\vert^{2}+\frac{c_{k}}{r^{2}}\left\vert\psi\right\vert^{2}
=\left(k^{2}+c_{k}\right)r^{2k-2}\left\vert w\right\vert^{2}
+2k\,r^{2k-1}ww'+r^{2k}\left\vert w'\right\vert^{2}.
\]
For the middle term, an integration by parts gives
\[
\begin{aligned}
2k\int_{0}^{\infty}r^{2k-1}ww'\,e^{-\frac{2r^{\sigma}}{\sigma}}r^{N-1}dr
&=k\int_{0}^{\infty}r^{2k+N-2}\left(\left\vert w\right\vert^{2}\right)'
e^{-\frac{2r^{\sigma}}{\sigma}}dr\\
&=-k\int_{0}^{\infty}\left\vert w\right\vert^{2}\,
\frac{d}{dr}\left(r^{2k+N-2}e^{-\frac{2r^{\sigma}}{\sigma}}\right)dr\\
&=-k(2k+N-2)\int_{0}^{\infty}\left\vert w\right\vert^{2}r^{2k+N-3}
e^{-\frac{2r^{\sigma}}{\sigma}}dr\\
&\quad+2k\int_{0}^{\infty}\left\vert w\right\vert^{2}r^{2k+N-2+\sigma-1}
e^{-\frac{2r^{\sigma}}{\sigma}}dr.
\end{aligned}
\]
Since $\sigma-1=s$ and $\left\vert\psi\right\vert^{2}=r^{2k}\left\vert w\right\vert^{2}$, the last integral is exactly
$\int_{0}^{\infty}\left\vert\psi\right\vert^{2}e^{-\frac{2r^{\sigma}}{\sigma}}r^{N+s-2}dr$. Finally,
\[
k^{2}+c_{k}-k(2k+N-2)=k^{2}+k(N+k-2)-2k^{2}-k(N-2)=0,
\]
so that the two remaining terms cancel, and \eqref{factors} follows.
\end{proof}

We note that \eqref{factors} is an identity, and that it exhibits $2k$, that is
$D_{k}^{(0)}$ by \eqref{Dk2k}, as the sharp constant on the node $k$, the equality
occurring exactly when $w$ is constant, that is, when $\psi(r)=d\,r^{k}$.

The second building block is a Poincar\'e inequality for the one dimensional weight
$r^{M-1}e^{-2r^{\sigma}/\sigma}$, which we shall apply with $M=N$ on the radial node and
with $M=2k+N$ on the remainder of \eqref{factors}. It contains Theorem \ref{TD} as the case $\sigma=2$.

\begin{lemma}\label{Llaguerre}
Let $M>0$, let $\sigma>0$, and let $w$ be real-valued and smooth, with all the
integrals below finite. Set
\[
\begin{aligned}
\mathcal{A}[w]&:=\inf_{c}\int_{0}^{\infty}\left\vert w-c\right\vert^{2}
e^{-\frac{2r^{\sigma}}{\sigma}}r^{M+\sigma-3}dr,\\
\mathcal{B}[w]&:=\inf_{c,d}\int_{0}^{\infty}
\left\vert w-c-d\,r^{\sigma}\right\vert^{2}
e^{-\frac{2r^{\sigma}}{\sigma}}r^{M+\sigma-3}dr.
\end{aligned}
\]
Then
\begin{equation}\label{laguerreP}
\int_{0}^{\infty}\left\vert w'\right\vert^{2}
e^{-\frac{2r^{\sigma}}{\sigma}}r^{M-1}dr
\ \geq\ 2\sigma\,\mathcal{A}[w]+2\sigma\,\mathcal{B}[w],
\end{equation}
where the infima are taken over real $c,d$. Since $\mathcal{B}[w]\geq0$, the inequality
\eqref{laguerreP} contains in particular the weighted Poincar\'e inequality
\begin{equation}\label{laguerrePweak}
\int_{0}^{\infty}\left\vert w'\right\vert^{2}
e^{-\frac{2r^{\sigma}}{\sigma}}r^{M-1}dr
\ \geq\ 2\sigma\,\mathcal{A}[w].
\end{equation}
Both constants $2\sigma$ are sharp: in the completion of $C_{0}^{\infty}([0,\infty))$
under the norm on the left-hand side, the equality holds in \eqref{laguerreP} if and
only if $w=c+d\,r^{\sigma}+e\,r^{2\sigma}$, and in \eqref{laguerrePweak} if and only if
$w=c+d\,r^{\sigma}$.
\end{lemma}

\begin{proof}

Both inequalities, in a more general form, are contained in \cite{LLLS26}: with the
choice of parameters $\delta=\frac{2}{\sigma}$, $\tau=\sigma$ and $\beta=M-1$, for which
the measure $e^{-\delta r^{\tau}}r^{\beta}dr$ of \cite{LLLS26} becomes
$e^{-\frac{2r^{\sigma}}{\sigma}}r^{M-1}dr$, we have $\delta\tau^{2}=2\sigma$ and
$r^{\beta+\tau-2}=r^{M+\sigma-3}$, while the condition $\frac{\beta+\tau-1}{\tau}>0$
becomes $\gamma>0$, where
\[
\gamma:=\frac{\sigma+M-2}{\sigma}.
\]
With these substitutions, \eqref{laguerrePweak} is \cite[Theorem 2.3]{LLLS26} and
\eqref{laguerreP} is \cite[Theorem 6.1]{LLLS26}. We nevertheless give the short spectral
proof, since the same expansion will be used again in Remark \ref{Rspectral}, and since
it yields at the same time the equality cases.

Performing the change of variable $t=\frac{2r^{\sigma}}{\sigma}$
and writing $w(r)=W(t)$, we have
\[
\int_{0}^{\infty}\left\vert w'\right\vert^{2}
e^{-\frac{2r^{\sigma}}{\sigma}}r^{M-1}dr
=2\left(\frac{\sigma}{2}\right)^{\gamma}
\int_{0}^{\infty}\left\vert\dot{W}\right\vert^{2}t^{\gamma}e^{-t}\,dt,
\]
and, in the same way,
\[
\int_{0}^{\infty}\left\vert w-c\right\vert^{2}
e^{-\frac{2r^{\sigma}}{\sigma}}r^{M+\sigma-3}dr
=\frac{1}{2}\left(\frac{\sigma}{2}\right)^{\gamma-1}
\int_{0}^{\infty}\left\vert W-c\right\vert^{2}t^{\gamma-1}e^{-t}\,dt.
\]
The eigenvalues of the Laguerre operator
\[
\mathcal{L}W:=-\frac{1}{t^{\gamma-1}e^{-t}}
\frac{d}{dt}\left(t^{\gamma}e^{-t}\dot{W}\right)
\]
are the nonnegative integers $\lambda_{n}=n$, and the corresponding eigenfunctions are
the generalized Laguerre polynomials $L_{n}^{(\gamma-1)}$, which form an orthogonal
basis of $L^{2}\left(t^{\gamma-1}e^{-t}dt\right)$; see for instance
\cite[Chapter V]{Szego} or \cite[Section 2.2]{LLLS26}. Writing
$W=\sum_{n\geq0}a_{n}L_{n}^{(\gamma-1)}$ and denoting by $h_{n}>0$ the squared norms of
the $L_{n}^{(\gamma-1)}$, we obtain
\begin{equation}\label{laguerreexp}
\begin{aligned}
\int_{0}^{\infty}\left\vert w'\right\vert^{2}
e^{-\frac{2r^{\sigma}}{\sigma}}r^{M-1}dr
&=2\left(\frac{\sigma}{2}\right)^{\gamma}\sum_{n\geq0}n\,a_{n}^{2}h_{n},\\
\mathcal{A}[w]=\frac{1}{2}\left(\frac{\sigma}{2}\right)^{\gamma-1}
\sum_{n\geq1}a_{n}^{2}h_{n},
&\qquad
\mathcal{B}[w]=\frac{1}{2}\left(\frac{\sigma}{2}\right)^{\gamma-1}
\sum_{n\geq2}a_{n}^{2}h_{n}.
\end{aligned}
\end{equation}
The ratio of the two normalizing factors is exactly
$2\left(\frac{\sigma}{2}\right)^{\gamma}\big/
\frac{1}{2}\left(\frac{\sigma}{2}\right)^{\gamma-1}=2\sigma$. Hence \eqref{laguerreP} reads
$\sum_{n\geq2}(n-2)a_{n}^{2}h_{n}\geq0$, with equality if and only if $a_{n}=0$ for all
$n \geq 3$, and \eqref{laguerrePweak} reads
$\sum_{n\geq2}(n-1)a_{n}^{2}h_{n}\geq0$, with equality if and only if $a_{n}=0$ for all
$n \geq 2$. Undoing the change of variable, these give
$w=c+d\,r^{\sigma}+e\,r^{2\sigma}$ and $w=c+d\,r^{\sigma}$ respectively.
\end{proof}

\subsection{The chain on the Fourier side}\label{Sfourierchain}

\begin{theorem}\label{TfracPoincare}
Let $N \geq 2$ and $s>0$, and let $\alpha_{0}=\cdots=\alpha_{m-1}=2$,
$\alpha_{m}=2\{s\}$ and $\alpha_{m+1}=2\left(1-\{s\}\right)$. Then, for every
$v \in C_{0}^{\infty}(\mathbb{R}^{N})$, real- or complex-valued,
\begin{equation}\label{fracPoincare}
\int_{\mathbb{R}^{N}}|\nabla v|^{2}d\mu_{s}
\ \geq\ \sum_{j=0}^{m}\alpha_{j}
\inf_{P \in \mathcal{Y}_{j}}\int_{\mathbb{R}^{N}}|v-P|^{2}d\nu_{s}
\ +\ \alpha_{m+1}\inf_{\substack{P \in \mathcal{Y}_{m}\\ d \in \mathbb{C}}}
\int_{\mathbb{R}^{N}}\left|v-P-d|\xi|^{\sigma}\right|^{2}d\nu_{s}.
\end{equation}
\end{theorem}

\begin{proof}
We may assume that $v$ is real-valued, the complex case following by applying the
result to the real and the imaginary parts and adding. We decompose
$v(\xi)=\sum_{k\geq0}\sum_{i=1}^{N_{k}}v_{k,i}(r)\phi_{k,i}(\theta)$ in spherical
harmonics as in Subsection \ref{Ssph}, with $r=|\xi|$ and $\theta=\xi/|\xi|$. By
\eqref{sphgrad},
\[
\int_{\mathbb{R}^{N}}|\nabla v|^{2}d\mu_{s}
=\sum_{k,i}\underbrace{\int_{0}^{\infty}
\left(\left\vert v_{k,i}'\right\vert^{2}+\frac{c_{k}}{r^{2}}\left\vert v_{k,i}\right\vert^{2}\right)
e^{-\frac{2r^{\sigma}}{\sigma}}r^{N-1}dr}_{=:\mathcal{D}_{k,i}},
\]
and
\[
\int_{\mathbb{R}^{N}}|v|^{2}d\nu_{s}
=\sum_{k,i}\underbrace{\int_{0}^{\infty}\left\vert v_{k,i}\right\vert^{2}
e^{-\frac{2r^{\sigma}}{\sigma}}r^{N+s-2}dr}_{=:\mathcal{N}_{k,i}}.
\]
Moreover, since $|\xi|^{k}\phi_{k,i}(\theta)$ is the solid harmonic of degree $k$
attached to $\phi_{k,i}$, and since every $P \in \mathcal{Y}_{j}$ is uniquely a sum of
solid harmonics of degrees $0,\ldots,j$, we have, for $0 \leq j \leq m$,
\[
\inf_{P \in \mathcal{Y}_{j}}\int_{\mathbb{R}^{N}}|v-P|^{2}d\nu_{s}
=\sum_{k\leq j,\,i}E_{k,i}+\sum_{k>j,\,i}\mathcal{N}_{k,i},
\]
where
\[
E_{k,i}:=\inf_{d}\int_{0}^{\infty}\left\vert v_{k,i}-d\,r^{k}\right\vert^{2}
e^{-\frac{2r^{\sigma}}{\sigma}}r^{N+s-2}dr,
\]
and the last infimum in \eqref{fracPoincare} equals
$B_{0}+\sum_{1\leq k\leq m,\,i}E_{k,i}+\sum_{k>m,\,i}\mathcal{N}_{k,i}$, where
$B_{0}$ denotes the quantity $\mathcal{B}$ of Lemma \ref{Llaguerre} associated with
$v_{0,1}$. We also write $A_{0}:=E_{0,1}$.

It therefore suffices to check, node by node, that the coefficients
$\alpha_{0},\ldots,\alpha_{m+1}$ are admissible. We use the two building blocks.
On the radial node, \eqref{laguerreP} with $M=N$ gives
\[
\mathcal{D}_{0,1}\ \geq\ C_{1}A_{0}+C_{1}B_{0},
\qquad C_{1}:=2\sigma=2(1+s),
\]
so that the required conditions are $\sum_{j=0}^{m}\alpha_{j}\leq C_{1}$ and
$\alpha_{m+1}\leq C_{1}$. On the node $k \geq 1$, Lemma \ref{Lfacts} followed by
\eqref{laguerrePweak} with $M=2k+N$ gives
\[
\mathcal{D}_{k,i}\ \geq\ 2k\,\mathcal{N}_{k,i}+C_{1}E_{k,i},
\]
so that the required conditions are $\sum_{j<k}\alpha_{j}\leq 2k$ and
$\sum_{j\geq k}\alpha_{j}\leq C_{1}$ for $1 \leq k \leq m$, together with
$\sum_{j=0}^{m+1}\alpha_{j}\leq 2(m+1)$ for the nodes $k \geq m+1$.

All these conditions hold for our choice. Indeed, $\sum_{j<k}\alpha_{j}=2k$ for
$1 \leq k \leq m$, and $\sum_{j=0}^{m+1}\alpha_{j}=2m+2$. Moreover
$\sum_{j=0}^{m}\alpha_{j}=2m+2\{s\}=2\lfloor s\rfloor+2+2\{s\}=2(1+s)=C_{1}$, and
$\alpha_{m+1}=2(1-\{s\})\leq 2\leq C_{1}$. Finally, for $1 \leq k \leq m$,
\[
\sum_{j\geq k}\alpha_{j}=2(m+1)-2k\leq 2m\leq 2\lfloor s\rfloor+2=2(1+s)-2\{s\}
\leq C_{1}.
\]
This proves \eqref{fracPoincare}.
\end{proof}

\subsection{The line \texorpdfstring{$b=-1$}{b=-1}: proof of Theorem \ref{Tfrac2cascade}}
\label{Sbm1}

The two building blocks of Subsection \ref{Sblocks} were stated for the line $b=0$, but
they only used the structure of the problem, and they carry over. We indicate the
changes. For the parameters $(a,b)=(-s,-1)$ we have $1+b-a=s$, so that here
$\sigma=s$, and the two measures are
\[
d\widetilde{\mu}_{s}:=e^{-\frac{2|\xi|^{s}}{s}}|\xi|^{2}d\xi,
\qquad
d\widetilde{\nu}_{s}:=e^{-\frac{2|\xi|^{s}}{s}}|\xi|^{s}d\xi.
\]

\begin{lemma}\label{Lfacts2}
Let $k \geq 1$, let $\psi \in C_{0}^{\infty}((0,\infty))$ be real-valued, and write
$\psi(r)=r^{D_{k}/2}w(r)$. Then
\[
\begin{aligned}
&\int_{0}^{\infty}\left(\left\vert\psi'\right\vert^{2}+\frac{c_{k}}{r^{2}}\left\vert\psi\right\vert^{2}\right)
e^{-\frac{2r^{s}}{s}}r^{N+1}dr\\
&\qquad=D_{k}\int_{0}^{\infty}\left\vert\psi\right\vert^{2}e^{-\frac{2r^{s}}{s}}r^{N+s-1}dr
+\int_{0}^{\infty}\left\vert w'\right\vert^{2}e^{-\frac{2r^{s}}{s}}r^{D_{k}+N+1}dr.
\end{aligned}
\]
\end{lemma}

\begin{proof}
We repeat the computation of Lemma \ref{Lfacts} with $\delta:=D_{k}/2$ in place of $k$
and with the weight $r^{N+1}$ in place of $r^{N-1}$. The integration by parts now
produces
\[
-\delta(2\delta+N)\int_{0}^{\infty}\left\vert w\right\vert^{2}r^{2\delta+N-1}
e^{-\frac{2r^{s}}{s}}dr
+2\delta\int_{0}^{\infty}\left\vert w\right\vert^{2}r^{2\delta+N+s-1}e^{-\frac{2r^{s}}{s}}dr,
\]
and the second integral is $D_{k}\int_{0}^{\infty}\left\vert\psi\right\vert^{2}
e^{-2r^{s}/s}r^{N+s-1}dr$ since $2\delta=D_{k}$. The coefficient of the remaining term
is
\[
\delta^{2}+c_{k}-\delta(2\delta+N)=-\left(\delta^{2}+N\delta-c_{k}\right),
\]
which vanishes precisely because $\delta=\frac{\sqrt{N^{2}+4c_{k}}-N}{2}$ is the
positive root of $X^{2}+NX-c_{k}=0$, that is, precisely because $D_{k}$ is defined by
\eqref{Dk}.
\end{proof}

Lemma \ref{Llaguerre} with $\sigma=s$ gives, for every $M>0$,
\begin{equation}\label{laguerre2}
\int_{0}^{\infty}\left\vert w'\right\vert^{2}e^{-\frac{2r^{s}}{s}}r^{M-1}dr
\ \geq\ 2s\,\inf_{c}\int_{0}^{\infty}\left\vert w-c\right\vert^{2}
e^{-\frac{2r^{s}}{s}}r^{M+s-3}dr,
\end{equation}
which is \eqref{laguerrePweak}, together with the stronger form \eqref{laguerreP}, both
constants being $2s=C_{1}$. We apply \eqref{laguerreP} with $M=N+2$ on the radial node,
and \eqref{laguerre2} with
$M=D_{k}+N+2$ on the remainder of Lemma \ref{Lfacts2}; in the latter case
$r^{M+s-3}=r^{D_{k}+N+s-1}$, so that the resulting quantity is exactly
$\inf_{d}\int_{0}^{\infty}\left\vert\psi-d\,r^{D_{k}/2}\right\vert^{2}
e^{-2r^{s}/s}r^{N+s-1}dr$.

\begin{proof}[Proof of Theorem \ref{Tfrac2cascade}]
The two blocks just obtained read
\[
\mathcal{D}_{0,1}\geq C_{1}A_{0}+C_{1}B_{0},
\qquad
\mathcal{D}_{k,i}\geq D_{k}\mathcal{N}_{k,i}+C_{1}E_{k,i}
\quad (k\geq1),
\qquad C_{1}=2s,
\]
with the same notation as in the proof of Theorem \ref{TfracPoincare}, the only
difference being that the profile on the node $k$ is now $r^{D_{k}/2}$ instead of
$r^{k}$. Consequently, the coefficients $\beta_{0},\ldots,\beta_{m+1}$ are admissible
provided that
\[
\begin{aligned}
&\sum_{j=0}^{m}\beta_{j}\leq C_{1},
\qquad
\beta_{m+1}\leq C_{1},
\qquad
\sum_{j=0}^{m+1}\beta_{j}\leq D_{m+1},\\
&\sum_{j<k}\beta_{j}\leq D_{k}
\quad\text{and}\quad
\sum_{j\geq k}\beta_{j}\leq C_{1}
\qquad (1\leq k\leq m).
\end{aligned}
\]
With our choice we have $\sum_{j<k}\beta_{j}=D_{k}$ for $1\leq k\leq m$ and
$\sum_{j=0}^{m+1}\beta_{j}=D_{m+1}$, so the third and the fourth conditions hold with
equality. Moreover
$\sum_{j=0}^{m}\beta_{j}=\min\{2s,D_{m+1}\}\leq C_{1}$ and
$\beta_{m+1}=(D_{m+1}-2s)^{+}\leq D_{m+1}-D_{1}\leq 2s=C_{1}$ by \eqref{mcond}, where we
used $D_{1}\leq 2s$. Finally, for $1\leq k\leq m$,
\[
\sum_{j\geq k}\beta_{j}=D_{m+1}-D_{k}\leq D_{m+1}-D_{1}\leq 2s=C_{1},
\]
again by \eqref{mcond}. This proves the weighted Poincar\'e chain for
$d\widetilde{\mu}_{s}$ and $d\widetilde{\nu}_{s}$. Transferring it to the CKN deficit by
Proposition \ref{Pfourierstab} with the parameters $(a,b)=(-s,-1)$ and with the
dilation invariant sets $\mathcal{K}_{j}=\mathcal{Y}_{j}^{\ast}$ and
$\mathcal{K}_{m+1}=\mathcal{Y}_{m}^{\ast}+\mathbb{C}|\xi|^{s}$, which divides all the
constants by $2$, and then undoing the Fourier transform by \eqref{dictA} with
$2\kappa=s$ and recalling \eqref{Hcal}, we obtain \eqref{frac2cascade}.
\end{proof}

\begin{remark}\label{Rspectral}
The two building blocks of Subsection \ref{Sblocks} in fact combine into an
\emph{identity}, and this explains the shape of the coefficients in Theorem
\ref{Tfraccascade} and in Theorem \ref{Tfrac2cascade}. Let us describe this on the line
$b=-1$, the line $b=0$ being identical with $D_{k}$ replaced by $2k$, $C_{1}$ by
$2(1+s)$, and $s$ by $\sigma=1+s$ in $\gamma_{k}$ and in the argument of the Laguerre
polynomial. Keeping the notation of Lemma \ref{Lfacts2} for $k \geq 1$, and noting that
on the radial node $k=0$ one has $D_{0}=0$ and the factorization is trivial, that is,
$w=\psi$, the expansion \eqref{laguerreexp} gives, for every $k \geq 0$ and up to the
common normalizing factor
$\frac{1}{2}\left(\frac{\sigma}{2}\right)^{\gamma_{k}-1}$, which depends on $k$ but
cancels in the ratio,
\[
\mathcal{D}_{k,i}=\sum_{n\geq0}\left(D_{k}+n\,C_{1}\right)a_{n}^{2}h_{n},
\qquad
\mathcal{N}_{k,i}=\sum_{n\geq0}a_{n}^{2}h_{n}.
\]
In other words, the quadratic form on the left-hand side of \eqref{frac2cascade} is
diagonalized, on the Fourier side, by the functions
\[
\psi_{k,n,i}(\xi):=|\xi|^{\frac{D_{k}}{2}}
L_{n}^{(\gamma_{k}-1)}\!\left(\frac{2|\xi|^{s}}{s}\right)
\phi_{k,i}\!\left(\frac{\xi}{|\xi|}\right),
\qquad
\gamma_{k}:=\frac{s+D_{k}+N}{s},
\]
and the corresponding eigenvalues are
\begin{equation}\label{spectrum}
\Lambda_{k,n}=D_{k}+n\,C_{1},
\qquad k \geq 0,\ n \geq 0,
\end{equation}
with $D_{0}=0$, so that $\Lambda_{0,0}=0$ corresponds to the constants. Thus the
spherical harmonic degree $k$ contributes $D_{k}$, while each radial excitation costs
the same amount $C_{1}$, on every node.

This gives a transparent description of the chain. Let
$\Lambda^{(0)}<\Lambda^{(1)}<\cdots$ denote the eigenvalues \eqref{spectrum} arranged in
increasing order, and assume that $D_{m}<C_{1}<D_{m+1}$. Then, by the second condition
in \eqref{mcond}, which gives $D_{m+1}\leq D_{1}+C_{1}=\Lambda_{1,1}$, the smallest
$m+3$ eigenvalues are
\[
0<D_{1}<\cdots<D_{m}<C_{1}<D_{m+1};
\]
without that condition, $\Lambda_{1,1}$ would come before $D_{m+1}$ and this list would
be wrong. Moreover, for $0 \leq j \leq m+1$, the span of the eigenfunctions whose
eigenvalue is at most $\Lambda^{(j)}$ is exactly the competitor set of the $j$-th layer
of the corresponding weighted Poincar\'e chain; here we use that the span of $1$ and
$|\xi|^{s}$ is the span of $\psi_{0,0}$ and $\psi_{0,1}$. Consequently
\[
\beta_{j}=\Lambda^{(j+1)}-\Lambda^{(j)},
\qquad 0 \leq j \leq m+1.
\]
In particular, all the constants of the weighted Poincar\'e chains, that is, of
\eqref{fracPoincare}, of its analogue on the line $b=-1$ and of
\eqref{improvedPoincare}, are sharp: testing with an eigenfunction of eigenvalue
$\Lambda^{(j+1)}$, which is charged exactly
$\beta_{0}+\cdots+\beta_{j}=\Lambda^{(j+1)}$ by the right-hand side, shows that none of
the $\beta_{j}$ can be increased when the others are kept fixed.

The same test applies to \eqref{frac2cascade} and to \eqref{fraccascade} for every
$j \neq m$, for which the eigenfunction with eigenvalue $\Lambda^{(j+1)}$ is non-radial.
For $j=m$ we have $\Lambda^{(m+1)}=C_{1}=\Lambda_{0,1}$, and the corresponding
eigenfunction is $\psi_{0,1}$; since $\psi_{0,0}$ and $\psi_{0,1}$ span the tangent space
of the family $\left\{c\,e^{-\lambda|\xi|^{\sigma}/\sigma}\right\}$, the extra parameter
$\lambda>0$ in the competitor sets of \eqref{frac2cascade} and \eqref{fraccascade}
absorbs that direction, both sides then vanish to higher order, and the test is empty.
We note that on the line $b=0$ one has $\alpha_{m}=2\{s\}=0$ whenever $s$ is an integer,
so that in that case nothing is lost.

Finally, if $D_{m+1}\leq C_{1}$, which can occur when the second condition in
\eqref{mcond} is the binding one, then $\beta_{m+1}=0$ and the chain stops after
$m+1$ layers. In that case $\beta_{0},\ldots,\beta_{m}$ are still the successive gaps, and are
therefore still sharp; here $\beta_{m}=D_{m+1}-D_{m}$ is governed by the non-radial
eigenfunction $\psi_{m+1,0}$, so that all the listed constants are covered by the test
above. The chain, however, no longer exhausts the eigenvalues below $C_{1}$.
\end{remark}

When $s=2$ we have $C_{1}=4$, and by Lemma \ref{Lordering} the largest $m$ satisfying
\eqref{mcond} is $m=2$, since $D_{2}<4<D_{3}$ and $D_{3}-D_{1}<4$. The coefficients of
Theorem \ref{Tfrac2cascade} then become
\[
\beta_{0}=D_{1},\qquad \beta_{1}=D_{2}-D_{1},\qquad
\beta_{2}=4-D_{2},\qquad \beta_{3}=D_{3}-4,
\]
which are exactly the four constants of \eqref{coupled}. Theorem \ref{Tcascade} is
therefore the case $s=2$ of Theorem \ref{Tfrac2cascade}, made explicit in the physical
variable and strengthened to its coupled form.

\subsection{Proof of Theorem \ref{Tfraccascade} and of Corollary \ref{CHUPcascade}}

\begin{proof}[Proof of Theorem \ref{Tfraccascade}]
We apply Proposition \ref{Pfourierstab} with the parameters $(a,b)=(-s,0)$, for which
$d\mu_{a,b}=d\mu_{s}$ and $d\nu_{a,b}=d\nu_{s}$, and with the sets
$\mathcal{K}_{j}=\mathcal{Y}_{j}$ for $0\leq j\leq m$ and
$\mathcal{K}_{m+1}=\mathcal{Y}_{m}+\mathbb{C}|\xi|^{\sigma}$, which are dilation
invariant. The hypothesis of Proposition \ref{Pfourierstab} is exactly
\eqref{fracPoincare}, and we conclude that
\[
\delta_{2}(\widehat{u})\ \geq\ \sum_{j=0}^{m-1}
\inf_{P \in \mathcal{Y}_{j},\,\lambda>0}\int_{\mathbb{R}^{N}}
\left|\widehat{u}-P\,e^{-\frac{\lambda|\xi|^{\sigma}}{\sigma}}\right|^{2}
|\xi|^{s-1}d\xi\ +\ \cdots,
\]
the remaining two terms being the ones of \eqref{fraccascade}. Undoing the Fourier
transform exactly as in the proof of Theorem \ref{Tfrac1}, that is, using
\eqref{dictA} with $2\kappa=s-1$, and recalling \eqref{Gcal}, we obtain
\eqref{fraccascade}. For the last assertion, we note that
$\sum_{j=0}^{m+1}\frac{\alpha_{j}}{2}=m+1=\lfloor s\rfloor+2$, and that
$\alpha_{m}=0$ when $s$ is an integer, in which case the remaining $\lfloor s\rfloor+2$
constants are all equal to $1$.
\end{proof}

\begin{proof}[Proof of Corollary \ref{CHUPcascade}]
We take $s=1$, so that $\sigma=2$, $m=2$, $\{s\}=0$, and hence $\alpha_{0}=\alpha_{1}=2$,
$\alpha_{2}=0$ and $\alpha_{3}=2$. Moreover $(-\Delta)^{\frac{s-1}{4}}$ is the identity,
and \eqref{frac1} is the classical Heisenberg Uncertainty Principle, as observed after Theorem \ref{Tfrac1}. It remains to identify the three families. By the Hecke identity, for
a solid harmonic $Y_{k}$ of degree $k$ one has
\[
\mathcal{F}^{-1}\left(Y_{k}(\xi)e^{-\frac{\lambda|\xi|^{2}}{2}}\right)(x)
=c_{k,\lambda}\,Y_{k}\!\left(x\right)e^{-\frac{2\pi^{2}|x|^{2}}{\lambda}}
\]
for an explicit constant $c_{k,\lambda}\neq0$; see
\cite[Chapter IV, Theorem 3.4]{SW71}.
Consequently $\mathcal{G}_{0}^{1}=\mathcal{Q}_{0}=E_{\mathrm{HUP}}$ and
$\mathcal{G}_{1}^{1}=\mathcal{Q}_{1}$, since $\mathcal{Y}_{1}$ consists of the
polynomials of degree at most $1$. Finally,
$\mathcal{F}^{-1}\left(|\xi|^{2}e^{-\lambda|\xi|^{2}/2}\right)
=-\frac{1}{4\pi^{2}}\Delta\,\mathcal{F}^{-1}
\left(e^{-\lambda|\xi|^{2}/2}\right)$ is a polynomial of degree $2$ times a Gaussian,
and $\mathcal{Y}_{2}$ together with $|\xi|^{2}$ spans all the polynomials of degree at
most $2$; hence $\mathcal{G}_{3}^{1}=\mathcal{Q}_{2}$. This gives \eqref{HUPcascade}.
\end{proof}

\section[The second order HUP: proofs of Theorem \ref{Tcascade} and Corollary \ref{Ccurlfree}]{The second order Heisenberg Uncertainty Principle: proofs of Theorem \ref{Tcascade} and Corollary \ref{Ccurlfree}}
\label{Spoincare}

Throughout this section we write $d\nu:=e^{-|x|^{2}}|x|^{2}\,dx$. We first record, on
the Fourier side, the case $s=2$ of Theorem \ref{Tfrac2cascade} and its coupled form;
we then invert the Fourier transform and prove Theorem \ref{Tcascade} and Corollary
\ref{Ccurlfree}.

\subsection{Two chains of weighted Poincar\'e inequalities}\label{Schains}

We begin with the case $s=2$ of Theorem \ref{Tfrac2cascade}, written for the measure
$d\nu$. We state it separately since $d\nu$ is one of the Gaussian type measures of
\cite{DLLN26}, and the resulting chain is of independent interest. We also refer to Lam and Lu \cite{LL26} for complete hierarchies of stability estimates for the Gaussian Poincar\'e inequality and for the $L^2$-Poincar\'e inequalities on Euclidean balls.

\begin{theorem}\label{TimprovedPoincare}
Let $N \geq 2$ and set $d\nu:=e^{-|x|^{2}}|x|^{2}\,dx$. For every
$v \in C_{0}^{\infty}(\mathbb{R}^{N})$, real- or complex-valued,
\begin{align}\label{improvedPoincare}
\int_{\mathbb{R}^{N}} |\nabla v|^{2}\,d\nu
&\geq D_{1}\inf_{c}\int_{\mathbb{R}^{N}} |v-c|^{2}\,d\nu \nonumber\\
&\quad+(D_{2}-D_{1})\inf_{c,\mathbf{a}}\int_{\mathbb{R}^{N}}
\left|v-c-|x|^{\frac{D_{1}}{2}-1}\mathbf{a}\cdot x\right|^{2} d\nu \nonumber\\
&\quad+(4-D_{2})\inf_{c,\mathbf{a},\mathbf{b}}\int_{\mathbb{R}^{N}}
\left|v-c-|x|^{\frac{D_{1}}{2}-1}\mathbf{a}\cdot x
-|x|^{\frac{D_{2}}{2}}\mathbf{b}\cdot\mathbf{\Phi}_{2}\right|^{2} d\nu \\
&\quad+(D_{3}-4)\inf_{c,d,\mathbf{a},\mathbf{b}}\int_{\mathbb{R}^{N}}
\left|v-c-d|x|^{2}-|x|^{\frac{D_{1}}{2}-1}\mathbf{a}\cdot x
-|x|^{\frac{D_{2}}{2}}\mathbf{b}\cdot\mathbf{\Phi}_{2}\right|^{2} d\nu, \nonumber
\end{align}
where the infima are taken over $c,d \in \mathbb{R}$ (resp.\ $\mathbb{C}$),
$\mathbf{a} \in \mathbb{R}^{N}$ (resp.\ $\mathbb{C}^{N}$), and
$\mathbf{b} \in \mathbb{R}^{N_{2}}$ (resp.\ $\mathbb{C}^{N_{2}}$).
\end{theorem}

Let us clarify the relation between Theorem \ref{TimprovedPoincare} and \cite{DLLN26}.
When $(a,b)=(-2,-1)$, we have $C_{1}:=2(1+b-a)=4>D_{1}$ and $D_{2}<4$ by Lemma
\ref{Lordering}, and hence $\min\{C_{1}-D_{1},\,D_{2}-D_{1}\}=D_{2}-D_{1}$. Therefore,
the first two lines of \eqref{improvedPoincare} recover the corresponding statement in
\cite[Section 2]{DLLN26}. The third and the fourth layers are new, and they are the
reason why \eqref{improvedPoincare} is a genuine chain rather than a single
improvement: they are obtained by retaining, in the spherical harmonics decomposition,
the node $k=2$ and the second radial Laguerre mode, and by redistributing the two
building blocks of \cite{DLLN26}, namely the sharp radial Poincar\'e inequality and the
exact remainder on each non-radial node, over the four nodes simultaneously. Each of
the four constants is optimal on the node that governs it; see the end of this
subsection. Let us finally observe that the chain closes after exactly
four layers. Indeed, the radial budget is already exhausted at the third layer, since
$\beta_{0}+\beta_{1}+\beta_{2}=D_{1}+(D_{2}-D_{1})+(4-D_{2})=4=C_{1}$, and a fifth
layer, that is, the node $k=3$, would in addition require $D_{4}-D_{1}\leq C_{1}=4$,
which fails for every $N \geq 2$ by \eqref{ordering2}. In other words, the four terms of
\eqref{improvedPoincare} are all that the measure $e^{-|x|^{2}}|x|^{2}dx$ can give in
this way.

\begin{proof}
This is the case $s=2$ of the chain established in Subsection \ref{Sbm1}. Indeed, for
$(a,b)=(-2,-1)$ we have $\sigma=1+b-a=2$, and therefore
\[
d\widetilde{\mu}_{2}=e^{-|\xi|^{2}}|\xi|^{2}d\xi=d\nu,
\qquad
d\widetilde{\nu}_{2}=e^{-|\xi|^{2}}|\xi|^{2}d\xi=d\nu,
\]
so that the two measures of Subsection \ref{Sbm1} coincide with $d\nu$. Moreover
$C_{1}=2s=4$, and, as computed at the end of Subsection \ref{Sbm1}, the largest integer $m$
satisfying \eqref{mcond} is $m=2$ and the coefficients $\beta_{j}$ of Theorem
\ref{Tfrac2cascade} are $D_{1}$, $D_{2}-D_{1}$, $4-D_{2}$ and $D_{3}-4$.
Finally, the competitor sets are identified as follows. Since $D_{1}/2$ and $D_{2}/2$
are the exponents attached to the nodes $k=1$ and $k=2$, the solid profiles
$|x|^{D_{1}/2}\phi_{1,i}$ and $|x|^{D_{2}/2}\phi_{2,i}$ are, up to constants,
$|x|^{\frac{D_{1}}{2}-1}x_{i}$ and $|x|^{\frac{D_{2}}{2}}\phi_{2,i}$, so that
$\mathcal{Y}_{0}^{\ast}=\{c\}$,
$\mathcal{Y}_{1}^{\ast}=\{c+|x|^{\frac{D_{1}}{2}-1}\mathbf{a}\cdot x\}$ and
$\mathcal{Y}_{2}^{\ast}=\{c+|x|^{\frac{D_{1}}{2}-1}\mathbf{a}\cdot x
+|x|^{\frac{D_{2}}{2}}\mathbf{b}\cdot\mathbf{\Phi}_{2}\}$, while the last set adds
$d|x|^{\sigma}=d|x|^{2}$. This is exactly \eqref{improvedPoincare}.
\end{proof}

On each node the constant that we use is optimal; this is the case $s=2$ of Remark
\ref{Rspectral}. Indeed, on the radial node it is
$C_{1}=4$, which is sharp by Theorem \ref{TD}, that is, by the first eigenvalue of the
Laguerre operator in Lemma \ref{Llaguerre}; and on the node $k \geq 1$ it is $D_{k}$,
the equality in Lemma \ref{Lfacts2} holding exactly when $\psi=d\,r^{D_{k}/2}$. The
limitation is not in the constants, which are the successive spectral gaps of Remark
\ref{Rspectral}, but in the fact that the chain closes after four layers, as measured by
the two inequalities $D_{3}-4<D_{1}$ and $D_{3}-D_{2}<4$, both of which follow from
\eqref{ordering}.

\subsection{The coupled form of the chain}

In Theorem \ref{TimprovedPoincare}, each of the four remainder terms is measured with
its own choice of the parameters $c,d,\mathbf{a},\mathbf{b}$. We now
show that one single choice of the parameters can realize all the four infima at the
same time, provided that the two radial competitors are taken orthogonal to each other.
Throughout this section, $L_{k}^{N/2}$ denotes the generalized Laguerre polynomial of
parameter $N/2$, and we recall that
\begin{equation}\label{laguerre}
\int_{0}^{\infty}L_{j}^{\alpha}L_{k}^{\alpha}e^{-t}t^{\alpha}\,dt
=\delta_{jk}\frac{\Gamma(k+\alpha+1)}{k!},
\qquad
\frac{d}{dt}L_{k}^{\alpha}=-L_{k-1}^{\alpha+1},
\qquad
L_{1}^{\alpha}(t)=\alpha+1-t.
\end{equation}

\begin{theorem}\label{TcoupledPoincare}
Let $N \geq 2$ and $d\nu=e^{-|x|^{2}}|x|^{2}\,dx$. For every
$v \in C_{0}^{\infty}(\mathbb{R}^{N})$, real- or complex-valued, one has
\begin{align}\label{coupledPoincare}
\int_{\mathbb{R}^{N}}|\nabla v|^{2}\,d\nu
\geq \inf_{c,d,\mathbf{a},\mathbf{b}}\Bigg[\,
&D_{1}\int_{\mathbb{R}^{N}}|v-c|^{2}\,d\nu
+(D_{2}-D_{1})\int_{\mathbb{R}^{N}}
\left|v-c-|x|^{\frac{D_{1}}{2}-1}\mathbf{a}\cdot x\right|^{2}d\nu \nonumber\\
&+(4-D_{2})\int_{\mathbb{R}^{N}}
\left|v-c-|x|^{\frac{D_{1}}{2}-1}\mathbf{a}\cdot x
-|x|^{\frac{D_{2}}{2}}\mathbf{b}\cdot\mathbf{\Phi}_{2}\right|^{2}d\nu \\
&+(D_{3}-4)\int_{\mathbb{R}^{N}}
\left|v-c-d\,L_{1}^{N/2}\!\left(|x|^{2}\right) \right. \nonumber\\
&\hskip 2.2cm \left.
-|x|^{\frac{D_{1}}{2}-1}\mathbf{a}\cdot x
-|x|^{\frac{D_{2}}{2}}\mathbf{b}\cdot\mathbf{\Phi}_{2}\right|^{2}d\nu\Bigg].
\nonumber
\end{align}
Moreover, the right-hand side of \eqref{coupledPoincare} is equal to the right-hand
side of \eqref{improvedPoincare}. In other words, there is no loss at all when the
parameters are coupled in this way.
\end{theorem}

\begin{proof}
Since the infimum of a sum is always at least the sum of the infima,
\eqref{coupledPoincare} is stronger than \eqref{improvedPoincare}. Hence it is enough
to prove that the two right-hand sides are equal, and then \eqref{coupledPoincare}
follows from Theorem \ref{TimprovedPoincare}.

We decompose $v$ as in \eqref{sphdecomp} and we keep the notation
$\mathcal{N}_{k,i}, A_{0}, B_{0}, E_{k,i}$ from the proof of Theorem
\ref{TfracPoincare}, with the profile $r^{D_{k}/2}$ on the node $k$ in place of
$r^{k}$. On the radial node, we write $t=|x|^{2}$,
$\psi(t):=v_{0,1}(\sqrt{t})$ and $\psi=\sum_{k \geq 0}a_{k}L_{k}^{N/2}$. Then, by
\eqref{laguerre}, and with
$h_{k}:=\frac{|\mathbb{S}^{N-1}|}{2}\frac{\Gamma\left(k+\frac{N}{2}+1\right)}{k!}$, we
have
\begin{equation}\label{radialexpansion}
\int_{\mathbb{R}^{N}}\left|v_{0,1}-c-d\,L_{1}^{N/2}(|x|^{2})\right|^{2}d\nu
=|a_{0}-c|^{2}h_{0}+|a_{1}-d|^{2}h_{1}+\sum_{k \geq 2}|a_{k}|^{2}h_{k},
\end{equation}
and, in particular,
$\int_{\mathbb{R}^{N}}|v_{0,1}-c|^{2}d\nu
=|a_{0}-c|^{2}h_{0}+\sum_{k \geq 1}|a_{k}|^{2}h_{k}$. Similarly, on the node
$k \geq 1$, we expand $v_{k,i}=\sum_{j}b_{j}^{(k,i)}e_{j}$ in an orthogonal basis of
$L^{2}\left((0,\infty),e^{-r^{2}}r^{N+1}\,dr\right)$ whose first element is
$e_{0}=r^{D_{k}/2}$, and we get
\begin{equation}\label{nodeexpansion}
\int_{0}^{\infty}\left|v_{k,i}-\kappa\,r^{D_{k}/2}\right|^{2}
e^{-r^{2}}r^{N+1}\,dr
=\left|b_{0}^{(k,i)}-\kappa\right|^{2}\|e_{0}\|^{2}
+\sum_{j \geq 1}\left|b_{j}^{(k,i)}\right|^{2}\|e_{j}\|^{2}.
\end{equation}

Now we evaluate the bracket in \eqref{coupledPoincare} at a fixed
$(c,d,\mathbf{a},\mathbf{b})$. By \eqref{radialexpansion} and \eqref{nodeexpansion},
each of the four integrals is a sum over the nodes of terms of the form
\[
\left|\text{(coefficient of $v$)}-\text{(parameter)}\right|^{2}
\times\text{(weight)}
\ +\ \text{(terms free of the parameters)}.
\]
The key point is that each parameter is always paired with one and the same expansion
coefficient of $v$. Indeed, the parameter $c$ is always paired with $a_{0}$. This is
exactly where the orthogonality $L_{0}^{N/2} \perp L_{1}^{N/2}$ is used: if we had
written the last competitor as $c+d|x|^{2}$, then $c$ would be paired with $a_{0}$ in
the first three integrals, but with $a_{0}+a_{1}\left(\frac{N}{2}+1\right)$ in the
fourth one. Similarly, the parameter $d$ is always paired with $a_{1}$, the $j$-th
entry of $\mathbf{a}$ is always paired with $b_{0}^{(1,j)}$, and the $i$-th entry of
$\mathbf{b}$ is always paired with $b_{0}^{(2,i)}$. Consequently, as a function of
$(c,d,\mathbf{a},\mathbf{b})$, the bracket is a sum of nonnegative quadratic terms, and
all of them attain their minimum at the same point, namely
\[
c=a_{0},\qquad d=a_{1},\qquad
(\mathbf{a})_{j}=b_{0}^{(1,j)},\qquad
(\mathbf{b})_{i}=b_{0}^{(2,i)}.
\]
Moreover, this common minimizer realizes each of the four infima in
\eqref{improvedPoincare} separately. Therefore the two right-hand sides are equal.
\end{proof}

\begin{remark}\label{Rnaive}
The orthogonal parametrization is essential here. Suppose that we keep the naive family
$c+d|x|^{2}$ and that we couple its constant term with the parameter $c$ of the first
integral, that is, suppose that we ask for
\begin{equation}\label{claim}
\int_{\mathbb{R}^{N}}|\nabla v_{0,1}|^{2}d\nu
\geq 4\inf_{c,d \in \mathbb{R}}
\left(\int_{\mathbb{R}^{N}}|v_{0,1}-c|^{2}d\nu
+\int_{\mathbb{R}^{N}}\left|v_{0,1}-c-d|x|^{2}\right|^{2}d\nu\right).
\end{equation}
Then \eqref{claim} is false for every $N \geq 2$. Indeed, with the notation above,
\eqref{laguerre} gives
\[
\int_{\mathbb{R}^{N}}|\nabla v_{0,1}|^{2}d\nu
=2|\mathbb{S}^{N-1}|\sum_{k \geq 1}
\frac{|a_{k}|^{2}\,\Gamma\left(k+\frac{N}{2}+1\right)}{(k-1)!},
\]
whence
\begin{equation}\label{gap}
\int_{\mathbb{R}^{N}}|\nabla v_{0,1}|^{2}d\nu-4A_{0}-4B_{0}
=4\sum_{k \geq 3}(k-2)|a_{k}|^{2}h_{k}.
\end{equation}
On the other hand, we write $c+dt=c'L_{0}^{N/2}+d'L_{1}^{N/2}$, so that
$c=c'+d'm_{1}$ with $m_{1}=\frac{N}{2}+1$. Then, using $h_{1}=m_{1}h_{0}$, a direct
minimization of
\[
(c',d') \longmapsto
|a_{0}-c'-d'm_{1}|^{2}h_{0}+|a_{0}-c'|^{2}h_{0}+|a_{1}-d'|^{2}h_{1}
\]
gives the value $h_{0}\frac{m_{1}^{2}}{m_{1}+2}a_{1}^{2}$, which is attained at
$a_{0}-c'=a_{1}-d'=\frac{a_{1}m_{1}}{m_{1}+2}$. Therefore,
\begin{equation}\label{coupledgap}
4\inf_{c,d}\left(\int_{\mathbb{R}^{N}}|v_{0,1}-c|^{2}d\nu
+\int_{\mathbb{R}^{N}}\left|v_{0,1}-c-d|x|^{2}\right|^{2}d\nu\right)
-4A_{0}-4B_{0}
=4h_{0}\frac{m_{1}^{2}}{m_{1}+2}a_{1}^{2}.
\end{equation}
Now, if we choose $a_{k}=0$ for all $k \geq 3$ and $a_{1} \neq 0$, then the left-hand
side of \eqref{gap} is zero, while the right-hand side of \eqref{coupledgap} is
strictly positive. Hence \eqref{claim} cannot hold. We note that the quantity
$4h_{0}m_{1}^{2}a_{1}^{2}/(m_{1}+2)$ measures exactly the non-orthogonality of $1$ and
$|x|^{2}$ in $L^{2}(d\nu)$, and this is precisely what is removed by the substitution
$|x|^{2} \rightsquigarrow L_{1}^{N/2}(|x|^{2})$.
\end{remark}

\subsection{Inverting the Fourier transform}\label{Sinversion}

We now compute the inverse Fourier transforms of the four competitor families that
appear in Theorems \ref{TimprovedPoincare} and \ref{TcoupledPoincare}, that is, we make
Theorem \ref{Tfrac2cascade} explicit in the case $s=2$. For this purpose, we will use
the
Hecke--Bochner formula; see, for instance, \cite[Chapter IV, Theorem 3.10]{SW71}. It
says that if $P$ is a solid harmonic of degree $m$ on $\mathbb{R}^{N}$ and $f$ is
radial, then
\begin{equation}\label{heckebochner}
\mathcal{F}\left(Pf\right)(x)=i^{-m}P(x)\,H(|x|),
\qquad
H(\rho)=2\pi\rho^{-\left(\frac{N}{2}+m-1\right)}\int_{0}^{\infty}f(t)
J_{\frac{N}{2}+m-1}(2\pi\rho t)\,t^{\frac{N}{2}+m}\,dt,
\end{equation}
where $J_{\nu}$ is the Bessel function of the first kind.

\begin{lemma}\label{LinvFourier}
Let $\lambda>0$ and put $\beta:=2\pi^{2}/\lambda$. Then the following hold.
\begin{enumerate}
\item[(i)] $\displaystyle
\mathcal{F}^{-1}\left(e^{-\frac{\lambda|\xi|^{2}}{2}}\right)(x)
=\left(\frac{2\pi}{\lambda}\right)^{N/2}e^{-\beta|x|^{2}}$.

\item[(ii)] $\displaystyle
\mathcal{F}^{-1}\left(|\xi|^{2}e^{-\frac{\lambda|\xi|^{2}}{2}}\right)(x)
=\left(\frac{2\pi}{\lambda}\right)^{N/2}
\left(\frac{N}{\lambda}-\frac{4\pi^{2}}{\lambda^{2}}|x|^{2}\right)e^{-\beta|x|^{2}}$.

\item[(iii)] For $\mathbf{a} \in \mathbb{C}^{N}$,
\[
\mathcal{F}^{-1}\left((\mathbf{a}\cdot\xi)\,|\xi|^{\frac{D_{1}}{2}-1}
e^{-\frac{\lambda|\xi|^{2}}{2}}\right)(x)
=i\,A(\lambda)\,
{}_{1}F_{1}\!\left(\frac{D_{1}}{4}+\frac{N+1}{2},\ \frac{N+2}{2};\
-\beta|x|^{2}\right)\mathbf{a}\cdot x,
\]
where $\displaystyle
A(\lambda)=\left(\frac{2}{\lambda}\right)^{\frac{D_{1}}{4}+\frac{N+1}{2}}
\pi^{\frac{N+2}{2}}
\frac{\Gamma\left(\frac{D_{1}}{4}+\frac{N+1}{2}\right)}
{\Gamma\left(\frac{N}{2}+1\right)}$.

\item[(iv)] For $\mathbf{b} \in \mathbb{C}^{N_{2}}$,
\[
\begin{aligned}
&\mathcal{F}^{-1}\left(\left(\mathbf{b}\cdot\mathbf{\Phi}_{2}\right)
|\xi|^{\frac{D_{2}}{2}}e^{-\frac{\lambda|\xi|^{2}}{2}}\right)(x)\\
&\qquad=-B(\lambda)\,
{}_{1}F_{1}\!\left(\frac{D_{2}}{4}+\frac{N}{2}+1,\ \frac{N}{2}+2;\
-\beta|x|^{2}\right)|x|^{2}\,\mathbf{b}\cdot\mathbf{\Phi}_{2}(x),
\end{aligned}
\]
where $\displaystyle
B(\lambda)=\left(\frac{2}{\lambda}\right)^{\frac{D_{2}}{4}+\frac{N}{2}+1}
\pi^{\frac{N}{2}+2}
\frac{\Gamma\left(\frac{D_{2}}{4}+\frac{N}{2}+1\right)}
{\Gamma\left(\frac{N}{2}+2\right)}$.
\end{enumerate}
\end{lemma}

\begin{proof}
Part (i) is the classical Gaussian computation, and (ii) follows from (i) together
with $\mathcal{F}^{-1}(|\xi|^{2}g)
=-\frac{1}{4\pi^{2}}\Delta_{x}\mathcal{F}^{-1}(g)$.

For (iii), the function $P(\xi)=\mathbf{a}\cdot\xi$ is a solid harmonic of degree
$m=1$, and $f(t)=t^{\frac{D_{1}}{2}-1}e^{-\lambda t^{2}/2}$ is radial. We apply
\eqref{heckebochner}, together with $\mathcal{F}^{-1}g(x)=\mathcal{F}g(-x)$, which only
changes the sign of the odd factor $\mathbf{a}\cdot x$. Then, expanding
\[
J_{N/2}(z)=\sum_{k=0}^{\infty}
\frac{(-1)^{k}(z/2)^{2k+\frac{N}{2}}}{k!\,\Gamma\left(k+\frac{N}{2}+1\right)},
\]
we get, with $\rho=|x|$,
\[
H(\rho)=2\pi\rho^{-\frac{N}{2}}\sum_{k=0}^{\infty}
\frac{(-1)^{k}(\pi\rho)^{2k+\frac{N}{2}}}
{k!\,\Gamma\left(k+\frac{N}{2}+1\right)}
\int_{0}^{\infty}t^{\frac{D_{1}}{2}+N+2k}
e^{-\frac{\lambda t^{2}}{2}}\,dt.
\]
Since $\int_{0}^{\infty}t^{2\kappa-1}e^{-\lambda t^{2}/2}\,dt
=\frac{1}{2}\left(\frac{2}{\lambda}\right)^{\kappa}\Gamma(\kappa)$, the integral inside
is equal to
\[
\frac{1}{2}\left(\frac{2}{\lambda}\right)^{\frac{D_{1}}{4}+\frac{N+1}{2}+k}
\Gamma\left(\frac{D_{1}}{4}+\frac{N+1}{2}+k\right).
\]
Now, using
$(\kappa)_{k}=\Gamma(\kappa+k)/\Gamma(\kappa)$ and
$\Gamma\left(k+\frac{N}{2}+1\right)
=\Gamma\left(\frac{N}{2}+1\right)\left(\frac{N}{2}+1\right)_{k}$, we recognize the
series of
${}_{1}F_{1}\left(\frac{D_{1}}{4}+\frac{N+1}{2},\frac{N}{2}+1;
-\frac{2\pi^{2}\rho^{2}}{\lambda}\right)$, and this gives (iii).

For (iv), the function
$P(\xi)=\mathbf{b}\cdot\left(|\xi|^{2}\mathbf{\Phi}_{2}(\xi)\right)$ is a solid harmonic of
degree $m=2$, and we write
$\left(\mathbf{b}\cdot\mathbf{\Phi}_{2}\right)|\xi|^{\frac{D_{2}}{2}}
e^{-\frac{\lambda|\xi|^{2}}{2}}=P(\xi)f(|\xi|)$ with
$f(t)=t^{\frac{D_{2}}{2}-2}e^{-\lambda t^{2}/2}$. The formula \eqref{heckebochner}
with $m=2$, so that the Bessel index is $\nu=\frac{N}{2}+1$, then gives, after the same
expansion,
\[
H(\rho)=-\left(\frac{2}{\lambda}\right)^{\frac{D_{2}}{4}+\frac{N}{2}+1}
\pi^{\frac{N}{2}+2}
\frac{\Gamma\left(\frac{D_{2}}{4}+\frac{N}{2}+1\right)}
{\Gamma\left(\frac{N}{2}+2\right)}
{}_{1}F_{1}\!\left(\frac{D_{2}}{4}+\frac{N}{2}+1,\ \frac{N}{2}+2;\
-\frac{2\pi^{2}\rho^{2}}{\lambda}\right).
\]
Here the sign comes from $i^{-2}=-1$, and since $P$ is even, replacing $x$ by $-x$ does
not change anything.
\end{proof}

\subsection{Proofs of Theorem \ref{Tcascade} and of Corollary \ref{Ccurlfree}}

\begin{proof}[Proof of Theorem \ref{Tcascade}]
Let $u \in X$. We apply Proposition \ref{Pfourierstab} with $(a,b)=(-2,-1)$ to the chain
\eqref{coupledPoincare} of Theorem \ref{TcoupledPoincare}, following its proof but
without taking the infimum over $\lambda$ in the last step, and we combine the result
with the proof of Theorem \ref{TC}. In this way we obtain, with one single
$\lambda=\mu$, where $\mu$ is as in Lemma \ref{LCKNidentity}, and one single set of
parameters $c,d,\mathbf{a},\mathbf{b}$, the lower bound
\[
\begin{aligned}
\delta_{2}(\widehat{u}) \geq \inf\Bigg[
&\frac{D_{1}}{2}\int_{\mathbb{R}^{N}}
\left|\widehat{u}-\Pi_{0}\right|^{2}|\xi|^{2}\,d\xi
+\frac{D_{2}-D_{1}}{2}\int_{\mathbb{R}^{N}}
\left|\widehat{u}-\Pi_{1}\right|^{2}|\xi|^{2}\,d\xi\\
&+\frac{4-D_{2}}{2}\int_{\mathbb{R}^{N}}
\left|\widehat{u}-\Pi_{2}\right|^{2}|\xi|^{2}\,d\xi
+\frac{D_{3}-4}{2}\int_{\mathbb{R}^{N}}
\left|\widehat{u}-\Pi_{3}\right|^{2}|\xi|^{2}\,d\xi\Bigg],
\end{aligned}
\]
where
\[
\begin{aligned}
\Pi_{0}&=c\,e^{-\frac{\lambda|\xi|^{2}}{2}},
&\Pi_{1}&=\Pi_{0}+|\xi|^{\frac{D_{1}}{2}-1}\mathbf{a}\cdot\xi\,
e^{-\frac{\lambda|\xi|^{2}}{2}},\\
\Pi_{2}&=\Pi_{1}+|\xi|^{\frac{D_{2}}{2}}\mathbf{b}\cdot\mathbf{\Phi}_{2}\,
e^{-\frac{\lambda|\xi|^{2}}{2}},
&\Pi_{3}&=\Pi_{2}+d\,L_{1}^{N/2}\!\left(\lambda|\xi|^{2}\right)
e^{-\frac{\lambda|\xi|^{2}}{2}}.
\end{aligned}
\]
Moreover, by \eqref{dictA} with $\kappa=1$, for any $F \in L^{2}$ we have
\[
(2\pi)^{2}\int_{\mathbb{R}^{N}}
\left|\widehat{u}(\xi)-F(\xi)\right|^{2}|\xi|^{2}\,d\xi
=\int_{\mathbb{R}^{N}}
\left|\nabla\left(u-\mathcal{F}^{-1}F\right)\right|^{2}dx,
\]
and therefore it remains to identify $\mathcal{F}^{-1}\Pi_{j}$. By Lemma
\ref{LinvFourier}, writing $\beta=2\pi^{2}/\lambda>0$ and absorbing the nonzero
constants $\left(\frac{2\pi}{\lambda}\right)^{N/2}$, $A(\lambda)$ and $B(\lambda)$ into
the free parameters, we obtain
\[
\begin{aligned}
\mathcal{F}^{-1}\left(c\,e^{-\frac{\lambda|\xi|^{2}}{2}}\right)
&=\alpha e^{-\beta|x|^{2}},\\
\mathcal{F}^{-1}\left(|\xi|^{\frac{D_{1}}{2}-1}\mathbf{a}\cdot\xi\,
e^{-\frac{\lambda|\xi|^{2}}{2}}\right)
&=\Theta_{1}(x;\beta)\,\mathbf{a}\cdot x,\\
\mathcal{F}^{-1}\left(|\xi|^{\frac{D_{2}}{2}}\mathbf{b}\cdot\mathbf{\Phi}_{2}\,
e^{-\frac{\lambda|\xi|^{2}}{2}}\right)
&=\Theta_{2}(x;\beta)\,|x|^{2}\,\mathbf{b}\cdot\mathbf{\Phi}_{2}(x),
\end{aligned}
\]
where $\Theta_{1}$ and $\Theta_{2}$ are as in \eqref{Theta}. For the last profile, Lemma
\ref{LinvFourier} (i)--(ii) together with $L_{1}^{N/2}(t)=\frac{N}{2}+1-t$ give
\[
\begin{aligned}
\mathcal{F}^{-1}\left(L_{1}^{N/2}\!\left(\lambda|\xi|^{2}\right)
e^{-\frac{\lambda|\xi|^{2}}{2}}\right)(x)
&=\left(\frac{2\pi}{\lambda}\right)^{N/2}
\left[\left(\frac{N}{2}+1\right)-N+\frac{4\pi^{2}}{\lambda}|x|^{2}\right]
e^{-\beta|x|^{2}}\\
&=\left(\frac{2\pi}{\lambda}\right)^{N/2}\mathcal{L}(x;\beta),
\end{aligned}
\]
where $\mathcal{L}$ is as in \eqref{Lprofile}. Absorbing again the nonzero normalizing
constants into $\alpha,\tau,\mathbf{a},\mathbf{b}$, we get
$\mathcal{F}^{-1}\Pi_{j}=w_{j}(\cdot;P)$ with
$P=(\alpha,\tau,\mathbf{a},\mathbf{b},\beta)$ as in \eqref{wj}, and this proves
\eqref{coupled}. We stress that no reality assumption on $u$ has been used here, because
Theorem \ref{TcoupledPoincare} and Proposition \ref{Pfourierstab} hold for
complex-valued functions and complex parameters.
Finally, if $u$ is real-valued, then
$w_{j}(\cdot;\overline{P})=\overline{w_{j}(\cdot;P)}$, where $\overline{P}$ is the
tuple obtained by conjugating $\alpha,\tau,\mathbf{a},\mathbf{b}$ and keeping $\beta$.
Therefore, denoting by $P_{\mathrm{re}}$ and $P_{\mathrm{im}}$ the tuples of the real
and of the imaginary parts of $\alpha,\tau,\mathbf{a},\mathbf{b}$, the rate $\beta$
being kept in both, we have
\[
\left\|\nabla\left(u-w_{j}(\cdot;P)\right)\right\|_{2}^{2}
=\left\|\nabla\left(u-w_{j}(\cdot;P_{\mathrm{re}})\right)\right\|_{2}^{2}
+\left\|\nabla w_{j}(\cdot;P_{\mathrm{im}})\right\|_{2}^{2}
\]
for all $j$ at the same time, and hence the infimum can be restricted to the real $P$.
\end{proof}

\begin{proof}[Proof of Corollary \ref{Ccurlfree}]
A smooth curl-free vector field on $\mathbb{R}^{N}$ is a gradient, that is,
$\mathbf{U}=\nabla u$. Then, after an integration by parts,
$\int_{\mathbb{R}^{N}}|\nabla\mathbf{U}|^{2}dx=\int_{\mathbb{R}^{N}}|\Delta u|^{2}dx$,
and hence $\delta_{\mathrm{CF}}(\mathbf{U})=\delta_{\mathrm{S}}(u)$ and
$\|\mathbf{U}-\nabla w\|_{2}^{2}=\int_{\mathbb{R}^{N}}|\nabla(u-w)|^{2}dx$ for every
$w$. Therefore \eqref{curlcoupled} is exactly \eqref{coupled}. The sharpness of the
constant $\frac{D_{1}}{2}$ follows from \cite{DoLL26}; see also \cite{HY25}.
\end{proof}

%---------------------------------------------------------------------------------------------
\subsection*{Acknowledgements}
 
N. Lam was partially supported by an NSERC Discovery Grant. G. Lu were partially supported by grants from the Simons Foundation. V. H. Nguyen was supported
by the Vietnam National Foundation for Science and Technology Development (NAFOSTED)
[101.02-2025.33].

%---------------------------------------------------------------------------------------------

\end{document}